\documentclass[twoside]{article}
\usepackage{amsmath,amssymb,amscd,amsthm,verbatim,alltt,amsfonts,array}
\usepackage{mathrsfs}
\usepackage[english]{babel}
\usepackage{latexsym}
\usepackage{amssymb}
\usepackage{euscript}
\usepackage{graphicx}

\newtheorem{theorem}{Theorem}
\theoremstyle{plain}

\newtheorem{example}{Example}

\newtheorem{lemma}{Lemma}

\newtheorem{remark}{Remark}

\numberwithin{equation}{section}

\DeclareMathOperator{\dif }{d}
\DeclareMathOperator{\ord }{Ord}
\title{ A class of difference schemes uniformly convergent on a modified Bakhvalov mesh}
\author{Samir Karasulji{\'c}\footnote{corresponding author}, Helena Zarin, Enes  Duvnjakovi{\'c}}
\date{}

\begin{document}

\maketitle

\pagestyle{myheadings}
\markboth{A class of difference schemes\ldots}{S. Karasulji{\'c}, H. Zarin, E. Duvnjakovi{\'c}}

\begin{abstract}
	In this paper we consider the numerical solution of a singularly perturbed one-dimensional semilinear reaction-diffusion problem.

A class of differential schemes is constructed. There is a proof of the existence and uniqueness of the numerical solution for this constructed class of differential schemes. 
The central result of the paper is an $\varepsilon$--uniform convergence of the second order $\mathcal{O}\left(1/N^2 \right),$ for the discrete approximate solution on the modified Bakhvalov mesh. At the end of the paper there are numerical experiments, two representatives of the class of differential schemes are tested and it is shown the robustness of the method and concurrence of theoretical and experimental results.
\end{abstract}

\begin{quotation}

\end{quotation}

\thispagestyle{empty}


\section{Introduction}
 \label{sekcija1}
 \setcounter{theorem}{0}

 We consider the boundary value problem
\begin{equation}
\varepsilon^2y''(x)=f(x,y)\:\text{ on }\:\left[ 0,1\right],
\label{uvod1}
\end{equation}
\begin{equation}
y(0)=0,\:y(1)=0,
\label{uvod2}
\end{equation}
where $0<\varepsilon<1$ is a perturbation parameter  and  $f$ is a non-linear function. We assume that the nonlinear function $f$ is continuously differentiable, 
i.e. for $k\geqslant 2,\:f\in C^{k}\left([0,1]\times\mathbb{R} \right),$ and that it has a strictly positive derivative with respect to $y$
\begin{equation}
\dfrac{\partial f}{\partial y}=f_y\geq m>0\:\text{ on }\:\left[0,1\right]\times \mathbb{R}\:\:(m=const).
\label{uvod3}
\end{equation}
The boundary value problem  \eqref{uvod1}--\eqref{uvod2} under the condition  \eqref{uvod3} has a unique solution, see Lorenz \cite{lorenz1982stability}.
Differential equations with the small parameter $\varepsilon$ multiplying the highest order derivate terms are said to be singularly perturbed.

Singularly perturbed equations occur frequently in mathematical models of various areas of physics, chemistry, biology, engineering science, economics and even sociology. 
These equations appear in  analysis of practical applications, for example in fluid dinamics (aero and hydrodinamics), semiconductor theory, advection-dominated heat and mass transfer, theory of plates, shellsand chemical kinetics, seismology, geophysics, nonlinear mechanics and so on.

A common features of singularly perturbed  equations is that their solutions have tiny boundary or interior layers, in which there is a sudden change 
of the solution's values of these equations. Such sudden changes occur e.g. in physics when a viscous gas flows at high speed and has a contact with a solid surface, then in chemical reaction,  in which besides the reactants, a catalyst is also involved. 


Using classical numerical methods such are finite difference schemes and finite element methods, which do not take into account the appearance of the boundary or inner layer, we get results which are unacceptable from the standpoint of stability, the value of the error or the cost of calculation.

Our goal is to construct a numerical method to overcome the previously listed problems, i.e., to construct an $\varepsilon$--uniformly convergent numerical method for problem $\eqref{uvod1}-\eqref{uvod3}.$

The numerical method is said to be an $\varepsilon$--uniformly convergent in the maximum  discrete norm of the order $r$, if  
\[\left\|y-\overline{y}\right\|_{\infty}\leqslant C N^{-r}, \]
where $y$ is the exact solution of the original continuous problem, $\overline{y}$ is the numerical solution of a given continuous problem, $N$ is the number of mesh points, 
and $C$ is a constant which does not depend of $N$ nor $\varepsilon.$

Many authors have analyzed and made a great contribution to the study of the problem \eqref{uvod1}--\eqref{uvod3} with different assumptions about the function $f$; and as well as more general nonlinear problems. 

There were many constructed $\varepsilon$--uniformly convergent difference schemes of order 2 and higher (Herceg \cite{herceg1990}, Herceg and Surla \cite{herceg1991}, Herceg and Miloradovi{\'c} \cite{herceg2003}, Herceg and Herceg \cite{herceg2003a}, Kopteva and Lin\ss\, \cite{kopteva2001},  Kopteva and Stynes \cite{kopteva2001robust,kopteva2004}, Kopteva, Pickett and Purtill \cite{kopteva2009}, Lin\ss, Roos and Vulanovi{\'c} \cite{linss2000uniform}, Sun and Stynes \cite{stynes1996}, Stynes and Kopteva \cite{stynes2006numerical}, Surla and Uzelac \cite{surla2003}, Vulanovi{\'c} \cite{vulanovic1983, vulanovic1989, vulanovic1991second, vulanovic1993, vulanovic2004}, Kopteva \cite{kopteva2007maximum} etc.).

The numerical method which we are going to construct and analyze in this paper is a synthesis of the two approaches in numerical solving of the problem \eqref{uvod1}--\eqref{uvod3}, 
and in an adequate approximation of the given boundary problem and the use of a layer--adaptive mesh. As mentioned above, the exact solutions of the singular perturbation boundary value problems usually exhibits sharp boundary or interior layers. Our goal is to construct a scheme whose coefficients behave as similar as the exact solution in order to get the  best possible numerical results. 

Because that the coefficients of our scheme should be exponential type functions.

Construction and analysis of these exponentially fitting differences schemes for solving linear singular--perturbation problems  can be seen in Roos \cite{roos1990}, O'Riordan and Stynes \cite{riordan1986uniformly} etc, while the appropriate schemes for nonlinear problems can be seen in Niijima \cite{niijima}, O'Riordan and Stynes \cite{riordan1986uniform}, Stynes \cite{stynes1987adaptive} and others. 
Above mentioned the fitted exponential differential schemes are uniformly convergent. In order to obtain an $\varepsilon$--uniformly convergent method, we  need to use appropriate a layer-adapted mesh.

Shishkin mesh \cite{shishkin1988grid} and their modification \cite{surla1996,vulanovic2004,  linss2012approximation} and others, Bakhvalov mesh \cite{bahvalov1969} and their modification \cite{herceg1990, herceg1998, herceg2000, herceg2003, vulanovic1983, vulanovic1993} and others are the most used layer--adapted meshes. 

The method, appropriate for our purpose, was first presented by Boglaev \cite{boglaev1984approximate}, where the discretisation of the problem \eqref{uvod1}--\eqref{uvod3} on
a modified Bakhvalov mesh was analysed and first order uniform convergence with respect to $\varepsilon$ was demonstrated. 

Using the method of \cite{boglaev1984approximate}, authors constructed new difference schemes in papers \cite{samir2011scheme}  and  \cite{samir2010scheme} for the problem \eqref{uvod1}--\eqref{uvod3} and carried out numerical experiments.

In \cite{samir2015uniformlyconvergent, samir2015uniformly} authors constructed new difference schemes and proved the uniqueness of the numerical solution
and an $\varepsilon$--uniform convergence on a modified Shishkin mesh, and at the end presented numerical experiments, others results are in \cite{ samir2017construction, samir2018uniformly} and \cite{samir2011skoplje, samir2012uniformnly, samir2012class, samir2013collocation, samir2015construction}.

In order to obtain better results, instead of Shishkin mesh, we will use a modification of  Bakhvalov mesh.  We have decided to use the modification of Bakhvalov mesh constructed by  Vulanovi\'c \cite{vulanovic1993}. This mesh has the features that we need in our analysis of the value of the error numerical method.

Shishkin mesh  is much simpler than Bakhvalov mesh, but many difference schemes applied to Bakhvalov mesh show better results. In order to get  better results we used a modification of Bakhvalov mesh.

This paper consist of six parts and it has the following structure. The first part is Introduction.  Next, in Section 2 a class of differential schemes are constructed, and it is proven the theorem of existence and uniqueness of the numerical solution. Mesh construction is in Section 3. In Section 4, it is showed and proven the theorem of an
$\varepsilon$--uniform convergence. In Section 5 are numerical experiments which confirm the theoretical results. The last two sections are Conclusion and Acknowledgments.

We use $\mathbb{R}^{N+1}$ to denote the real $(N+1)$--dimensional linear space of all column vectors 
\[u=(u_0,u_1,\ldots,u_N)^T.\] 
We equip space $\mathbb{R}^{N+1}$ with usual maximum vector norm 
\[\left\| u\right\|_{\infty}=\max_{0\leqslant i\leqslant N}\left|u_i\right|.\] The induced norm of a linear mapping $A=(a_{ij}):\mathbb{R}^{N+1}\rightarrow\mathbb{R}^{N+1}$ is
\[\left\|A\right\|_{\infty}=\max_{0\leqslant i\leqslant N}\sum_{j=0}^N{\left|a_{ij}\right|}.\]

\begin{remark} Throughout this paper we let $C$, sometimes subscripted, denote a generic positive constant that may take different values in different formulas, but it is always independent of $N$ and $\varepsilon$.
\end{remark}

\section{Construction of the scheme}
 \label{sekcija2}
 \setcounter{theorem}{0}

We will use the well--known Green's function for the operator $L_{\varepsilon}y:=\varepsilon^2y''-\gamma y,$ for the construction of the differential scheme, where $\gamma$ is  a constant. The value of $\gamma$ will be determined later in this section.

This method, as we mentioned in Introduction, first was introduced by Boglaev in his paper  \cite{boglaev1984approximate}. 
Detailed construction of differential schemes done by this method, can be found in  \cite{samir2015uniformlyconvergent, samir2015uniformly}. In \cite{samir2015uniformly} was obtained the following  equality 
\begin{equation}
\begin{split}
&\frac{\beta}{\sinh(\beta h_{i-1})}  y_{i-1}-\left( \frac{\beta}{\tanh(\beta h_{i-1})}+\frac{\beta}{\tanh(\beta h_i)} \right)y_i
+\frac{\beta}{\sinh(\beta h_i)}y_{i+1}\\
&\hspace{4cm}=\dfrac{1}{\varepsilon^2}\left[\int\limits_{x_{i-1}}^{x_{i}}{u_{i-1}^{II}(s)\psi(s,y(s))ds}+\int\limits_{x_{i}}^{x_{i+1}}{u_{i}^{I}(s)\psi(s,y(s))ds}\right],\\
&\hspace{3cm}y_0=0,\:\:y_N=0,\:\:i=1,2,\cdots,N-1,
\end{split}
\label{konst13}
\end{equation}
where 

\begin{equation*}
 \psi(x,y(x))=f(x,y(x))-\gamma y(x),
\label{proizvoljna1}
\end{equation*}

\begin{equation}
   0=x_0<x_1<x_2<\cdots<x_N=1,
  \label{proizvoljna}
\end{equation} 
is an arbitrary mesh on $[0,1],$  $h_i=x_{i+1}-x_i,$   
\begin{equation}
 \beta=\dfrac{\sqrt{\gamma}}{\varepsilon},
 \label{beta}
\end{equation} 
functions $u^{I}_{i}$ and $u^{II}_{i}$ are the solutions of the next boundary value problem

\begin{align*}
\begin{array}{c}L_\varepsilon y=0\:\text{on}\left(x_{i},x_{i+1}\right),\\
	 u_i(x_i)=1,\: u_i(x_{i+1})=0,\\ i=0,1,...,N-1,
	 \end{array}
\text{\quad and \quad}
 \begin{array}{c}L_\varepsilon y=0\:\text{on}\left(x_{i},x_{i+1}\right),\\
	 u_i(x_i)=0,\:u_i(x_{i+1})=1,\\ i=0,1,...,N-1.
\end{array}
\label{dodatak2}
\end{align*}
and
\begin{equation*}
\begin{array}{c}u_i^{I}(x)=\dfrac{\sinh\left(\beta\left(x_{i+1}-x\right)\right)}{\sinh\left(\beta h_i \right)},\:u_i^{II}(x)=\dfrac{\sinh\left(\beta\left(x-x_i\right) \right)}{\sinh\left(\beta h_i\right)},\:x\in\left[ x_i,x_{i+1}\right],\\i=0,1,2,...,N-1.	
\end{array}
\label{dodatak6}
\end{equation*}

We cannot, in general, explicitly compute the integrals on the right-hand side of \eqref{konst13}. In order to get a simple enough difference scheme, we approximate the function $\psi$ on
$[x_{i-1},x_i]\cup[x_i,x_{i+1}]$ using 

\begin{equation}
\begin{array}{c}
\overline{\psi}_{i}=\dfrac{\psi(x_{i-1},\overline{y}_{i-1})+q\psi(x_{i},\overline{y}_{i})+\psi(x_{i+1},\overline{y}_{i+1})}{q+2},
\end{array}
\label{psi}
\end{equation}
where  $q\in\mathbb{R}^{+},$ while $\overline{y}_i$ are approximate values of the solution $y$ of the problem \eqref{uvod1}--\eqref{uvod3} at mesh points $x_i$. 

Finally, from \eqref{konst13}, using \eqref{psi}, we get the following difference scheme
\begin{equation}
 \begin{split} 
  \left[(q+1)a_{i}+d_{i}+\triangle d_{i+1}\right]&\left(\overline{y}_{i-1}-\overline{y}_i\right)
     -\left[(q+1)a_{i+1}+d_{i+1}+\triangle d_{i} \right]\left(\overline{y}_i-\overline{y}_{i+1}\right)\\
-&\dfrac{f(x_{i-1},\overline{y}_{i-1})+qf(x_{i},\overline{y}_{i})+f(x_{i+1},\overline{y}_{i+1}) }{\gamma}\left(\triangle d_{i}+\triangle d_{i+1}\right)=0,\\
\overline{y}_0=&0,\:\:\overline{y}_N=0,\:\:i=1,2,...,N-1,
 \end{split}
 \label{konst22}
\end{equation}
where $a_i=\frac{1}{\sinh(\beta h_{i-1})},\:d_i=\frac{1}{\tanh(\beta h_{i-1})} ,\:\triangle d_i= d_i-a_i.$

Let us introduce the discrete problem of the problem \eqref{uvod1}--\eqref{uvod3}, using \eqref{konst22} on the mesh \eqref{proizvoljna} we can write  
\begin{equation}
   F\overline{y}=\left(F_0\overline{y},F_1\overline{y},\ldots,F_N \overline{y}\right)^T=0,
  \label{diskretni1}
\end{equation}
where are

\begin{equation}
  \begin{split} 
   F_0\overline{y}:=&\overline{y}_0=0,\\   
   F_i\overline{y}:=&\frac{\gamma}{\triangle d_i+\triangle d_{i+1}}
          \left\{ \left[(q+1)a_{i}+d_{i}+\triangle d_{i+1}\right]\left(\overline{y}_{i-1}-\overline{y}_i\right) \right.\\
        &\hspace{.25cm} -\left[(q+1)a_{i+1}+d_{i+1}+\triangle d_{i} \right]\left(\overline{y}_i-\overline{y}_{i+1}\right)\\ 
        &\hspace{.5cm} -\dfrac{f(x_{i-1},\overline{y}_{i-1})+qf(x_{i},\overline{y}_{i})+f(x_{i+1},\overline{y}_{i+1}) }{\gamma}
        \left.\left(\triangle d_{i}+\triangle d_{i+1}\right)\right\},\\
        &\hspace{2cm}i=1,2,\ldots,N-1,\\
   F_N\overline{y}:=&\overline{y}_N=0.  
  \end{split}
 \label{diskretni2}
\end{equation}

\begin{theorem} The discrete problem $(\ref{diskretni1})$ for $\gamma\geq f_y,$ has the unique solution $\overline{y},$ where\\  $\overline{y}=(\overline{y}_0,\overline{y}_1,\overline{y}_2,...,\overline{y}_{N-1},\overline{y}_{N})^T\in\mathbb{R}^{N+1}.$ Moreover, for any $v,w\in\mathbb{R}^{N+1}$, the following stabilty inequality holds 
\begin{equation}
    \left\|w-v\right\|_{\infty}\leqslant \frac{1}{m}\left\|Fw-Fv\right\|_{\infty}.
  \label{stabilnost}
\end{equation}
\label{teorema21}
\end{theorem}
\begin{proof} We use a technique from \cite{herceg2003, vulanovic1993}, the proof of existence of the solution of $F\overline{y}=0$ is based on the proof of the following relation: $\left\| \left(F'\overline{y}\right)^{-1}\right\|_{\infty}\leq C,$ where $F'\overline{y}$ is a Fr\'echet derivative of $F$.

The Fr\'echet derivative $H:=F'\overline{y}$ is a tridiagonal matrix. Let $H=[h_{ij}].$  The non-zero elements of this tridiagonal matrix are
\begin{equation}
\begin{split}
h_{0,0}=&h_{N,N}=1,\\
h_{i,i}=&\tfrac{\gamma}{\triangle d_i+\triangle d_{i+1}}\left[ -q(a_i+a_{i+1})-2(d_{i}+ d_{i+1})
              -\tfrac{q}{\gamma}\cdot \tfrac{\partial f(x_{i},\overline{y}_{i})}{\partial \overline{y}_{i}}(\triangle d_i+\triangle d_{i+1}) \right]<0,\\
h_{i,i-1}=&\tfrac{\gamma}{\triangle d_i+\triangle d_{i+1}}\left[  \left(\triangle d_i+\triangle d_{i+1}\right)\left(1-\tfrac{1}{\gamma}\cdot\tfrac{\partial f(x_{i-1},\overline{y}_{i-1})}{\partial \overline{y}_{i-1}}\right)+(q+2) a_i \right]>0,\\
h_{i,i+1}=&\tfrac{\gamma}{\triangle d_{i}+\triangle d_{i+1}}\left[  \left(\triangle d_{i+1}+\triangle d_{i}\right)\left(1-\tfrac{1}{\gamma}\cdot\tfrac{\partial f(x_{i+1},\overline{y}_{i+1})}{\partial \overline{y}_{i+1}}\right)+(q+2) a_{i+1} \right]>0,\\
&\hspace{2cm}i=1,2,\ldots,N-1.
\end{split}
\label{konst24}
\end{equation}
Hence $H$ is an $L$--matrix.  Let us show that $H$ is an $M$--matrix. Now, we have 
\begin{multline}
|h_{i,i}|-|h_{i,i-1}|-|h_{i-1,i}|\\=\tfrac{\gamma}{\triangle d_i+\triangle d_{i+1}}
   \left[ \frac{(\triangle d_i+\triangle d_{i+1})\left(\tfrac{\partial f(x_{i-1},\overline{y}_{i-1})}{\partial \overline{y}_{i-1}} 
       +q\tfrac{\partial f(x_{i},\overline{y}_{i})}{\partial \overline{y}_{i}}
           +\tfrac{\partial f(x_{i+1},\overline{y}_{i+1})}{\partial \overline{y}_{i}}\right)}{\gamma}  \right]
   \geqslant(q+2)m.  
\label{konst25}
\end{multline}
Based on ($\ref{konst25}$), we have proved that $H$ is an $M$--matrix.  Since $H$ is an $M$--matrix, 
now we obtain
\begin{equation}
\left\|H^{-1}\right\|_{\infty}\leq \dfrac{1}{(q+2)m}.
\label{konst27}
\end{equation}
Finally, by the Hadamard Theorem (5.3.10 from \cite{ortega2000}), the first statement of our theorem follows. 

The second part of the proof is based on the part of the proof of \cite{herceg1990}. We have that 
\begin{equation}
    Fw-Fv=(F'u)(w-v),\:\text{ for some } u=(u_0,u_1,\ldots,u_N)^T\in\mathbb{R}^{N+1}. 
 \label{konst28}
\end{equation}
and
\begin{equation}
  w-v=\left(F'u\right)^{-1}\left( Fw-Fv\right).
 \label{konst29}
\end{equation}
Now, based on \eqref{konst27}, we have that 
\begin{multline}
  \left\|w-v\right\|_{\infty}=\left\|\left(F'u\right)^{-1}\left( Fw-Fv\right)\right\|_{\infty}\\
    \leqslant \frac{1}{(q+2)m}\left\|Fw-Fv\right\|_{\infty}\leqslant\frac{1}{m}\left\|Fw-Fv\right\|_{\infty} . 
\label{konst30}
\end{multline}
\end{proof}

\section{Mesh construction}
 \label{sekcija3}
 \setcounter{theorem}{0}

The exact solution $y$ of the problem \eqref{uvod1}--\eqref{uvod2} has boundary layers of exponential type near the points $x=0$ and $x=1.$  In order to achieve an $\varepsilon$--uniformly convergence of the numerical method, it is necessary to use a layer-adapted mesh. In the construction of this mesh  the occurrence of boundary or inner layers needs to be taken in account. We will use the modified Bakhvalov mesh from  \cite{vulanovic1993}, which has a sufficiently smooth generating function, that is going to  provide the necessary characteristics of the mesh that we need for further analysis.

The mesh $\triangle:x_0<x_1<...<x_N$ is generated by $x_i=\varphi(t_i),\,t_i=i/N=ih,\:h=1/N,\:i=0,1,\ldots,N;\,N=2m,\;m\in\mathbb{N}\backslash\left\{1\right\},$ with the mesh generating function 
\begin{equation}
  \begin{split} 
    \varphi(t)=
     \left\{  \begin{array}{l} 
         \kappa(t):=\tfrac{a \varepsilon t}{p-t},\:t\in\left[0,\alpha\right],\\
                \pi(t)=:\omega(t-\alpha)^3+\tfrac{\kappa''(\alpha)(t-\alpha)^2}{2}+\kappa'(\alpha)(t-\alpha)+\kappa(\alpha),
                           \:t\in\left[\alpha,1/2\right],\\
              1-\varphi(1-t),\:t\in\left[1/2,1\right],            
       \end{array}\right.
  \end{split}
 \label{bakh11}
\end{equation}
here $p$ is an arbitrary parameter from $\left(\left( \varepsilon^{\star}\right)^{1/3},1/2 \right),\,\varepsilon\in(0,\varepsilon^{\star}]$ and $\alpha=p-\varepsilon^{1/3}>0,$ where we assume 
that $\varepsilon^{\star}<\frac{1}{8}.$ The coefficient $\omega$ is determined from $\pi\left(\frac{1}{2} \right)=\frac{1}{2},$ we get
\[ \omega=\left(\tfrac{1}{2}-\alpha \right)^{-3}
 \left\{\tfrac{1}{2}-a\left[p\left( \tfrac{1}{2}-\alpha\right)^2+p\left(\tfrac{1}{2}-\alpha \right)\varepsilon^{1/3}+\alpha\varepsilon^{2/3} \right] \right\},\]
and $a$ is chosen such that $\omega\geqslant0,$ (such  $a,$ independent of $\varepsilon,$ obviously exist).

By this choice of $\alpha$  and $\pi$ we get

\begin{subequations}
  \begin{equation}
    \varphi\in C^{2}\left( \left[0,1\right]\setminus\left\{\tfrac{1}{2}\right\}\right),\vspace{.25cm}
    \label{bakh12}
  \end{equation}
 \begin{equation}
   \left| \varphi'(t)\right|\leqslant C,\:t\in\left[0,1\right],
   \label{bakh13}
 \end{equation}
 and
 \begin{equation}
   \left|\varphi''(t)\right|\leqslant C,\: t\in\left[ 0,1\right]\setminus\left\{\tfrac{1}{2}\right\}.
   \label{bakh14}
 \end{equation}
\end{subequations}

Values of the  mesh sizes $h_i,$ and values of differences  $h_{i+1}-h_i,$ will be given in the next lemma.

\begin{lemma} The mesh sizes $h_{i}=x_{i+1}-x_{i},$ defined by the generating function \eqref{bakh11}, satisfy
    \begin{subequations}
     \begin{equation}
      h_i\leqslant CN^{-1},\:i=0,1,\ldots,N-1,
     \label{korak1}     
     \end{equation}
     and
     \begin{equation}
     \left| h_{i}-h_{i-1}\right|\leqslant CN^{-2},\:i=1,2,\ldots,N-1.
     \label{korak2}
     \end{equation}
    \end{subequations}
\end{lemma} 
\begin{proof} Due to $\eqref{bakh13},$ we have  
 \begin{equation} 
   h_i=\int_{i/N}^{(i+1)/N}{\varphi'(t)\dif t}\leqslant C\int_{i/N}^{(i+1)/N}{\dif t}\leqslant C N^{-1}.
   \label{korak3}
 \end{equation}

Let us divide the proof of \eqref{korak2}, because of \eqref{bakh12},  into three parts. 
 
Firstly, when  $i\in\left\{1,\ldots,N-1 \right\}\backslash\left\{N/2-1,N/2,N/2+1\right\},$  based on \eqref{bakh14}, we have   
 \begin{equation}
  \left|h_{i}-h_{i-1} \right|=\left| \int_{i/N}^{(i+1)/N}\int_{t-1/N}^{t} {\varphi''(s)\dif s\dif t}\right|
                          \leqslant C \left| \int_{i/N}^{(i+1)/N}\int_{t-1/N}^{t} {\dif s\dif t}\right|\leqslant C N^{-2}.
  \label{korak4}                        
 \end{equation}

Secondly, for $i=N/2,$  we get  
 
\begin{equation}
  \begin{split} 
  \left|h_{N/2}-h_{N/2-1}\right|=&1-\varphi(1-(N/2+1)/N)-\frac{1}{2}-\left(\frac{1}{2}-\varphi((N/2-1)/N)\right)\\
                                =&\varphi((N/2-1)/N)-\varphi(1-(N/2+1)/N)\\
                                =&0,  
  \end{split}
  \label{korak5}  
\end{equation} 
and finally, for $i=N/2-1$ or $i=N/2+1,$ we get
\begin{equation}
  \begin{split} 
  |h_{N/2-1}-h_{N/2-2}|=|h_{N/2+1}-h_{N/2}|\leqslant\frac{6\omega}{N^3}+\frac{\mu''(\alpha)+(3-6\alpha)\omega}{N^2}\leqslant\frac{C}{N^2}.
  \end{split} 
\label{korak6}
\end{equation}

Now, using \eqref{korak3}, \eqref{korak4}, \eqref{korak5} and  \eqref{korak6} the inequalities   \eqref{korak1} and \eqref{korak2} are proven.
\end{proof}

\section{Uniform convergence}
 \label{sekcija4}
 \setcounter{theorem}{0}

In this section we prove the theorem on $\varepsilon$--uniform convergence of the discrete problem \eqref{diskretni1}. The proof of the theorem is based on relation $\left\|y-\overline{y}\right\|_{\infty}\leqslant C\left\|Fy-F\overline{y}\right\|_{\infty}.$ 

Stability of the differential sheme is proven in Theorem \ref{teorema21}, and as $F\overline{y}=0,$ it is enough to estimate the value of the expression $\left\|Fy\right\|_{\infty}.$

 The proof uses the decomposition of the solution $y$ to the problem (\ref{uvod1})--(\ref{uvod2}) to the layer $s$ and a regular component $r$, given in the following assertion. 

\begin{theorem} {\rm\cite{vulanovic1983}} The solution $y$ to problem {\rm\eqref{uvod1}--\eqref{uvod2}} can be represented in the following way:
\begin{equation*}
   y=r+s,
\end{equation*}
where for $j=0,1,\ldots,k+2$ and $x\in[0,1]$ we have
\begin{equation}
\left|r^{(j)}(x)\right|\leqslant C,
\label{regularna}
\end{equation}
and
\begin{equation}
\left|s^{(j)}(x)\right|\leqslant C \varepsilon^{-j}\left(e^{-\frac{x}{\varepsilon}\sqrt{m}}+e^{-\frac{1-x}{\varepsilon}\sqrt{m}}\right).
\label{slojna}
\end{equation}
\label{teorema2}
\end{theorem}
\begin{proof}
See in Vulanovi\' c \cite{vulanovic1983}.
\end{proof}

Note that $e^{-\frac{x}{\varepsilon}\sqrt{m}}\geqslant e^{-\frac{1-x}{\varepsilon}\sqrt{m}},\:\forall x\in[0,1/2]$ and $e^{-\frac{x}{\varepsilon}\sqrt{m}}\leqslant e^{-\frac{1-x}{\varepsilon}\sqrt{m}},\:\forall x\in[1/2,1].$  These inequalities and the estimate \eqref{slojna} imply that the analysis of the error value can be done for $\left\|Fy\right\|_{\infty}$ on the part of the mesh which corresponds to $[0,1/2]$ omitting the function $e^{-\frac{1-x}{\varepsilon}\sqrt{m}},$ keeping in mind that on this part of the mesh we have that $h_{i-1}\leqslant h_i.$ An analogous analysis holds for the part of the mesh which corresponds to $x\in[1/2,1],$ but with the omission of the function $e^{-\frac{x}{\varepsilon}\sqrt{m}}$ and using the inequality $h_{i-1}\geqslant h_i.$

In order to get as simpler analysis as possible, let us write $F_iy,\:i=1,2,\ldots,N-1,$ in the following form

\begin{equation}
  \begin{split} 
    F_iy
                      =&\gamma\left(y_{i-1}-2y_i+y_{i+1} \right)   \\
                                 &\hspace{.5cm}+\gamma \tilde{q} \frac{a_i(y_{i-1}-y_i)-a_{i+1}(y_i-y_{i+1})
                                     -\tfrac{f(x_i,y_i)}{\gamma}\left(\Delta d_i+\Delta d_{i+1}\right) }{\Delta d_i+\Delta d_{i+1}} \\
                                 &\hspace{1cm}+2\gamma\frac{a_i(y_{i-1}-y_i)-a_{i+1}(y_i-y_{i+1})}
                                     {\Delta d_i+\Delta d_{i+1}}  , \\
                                 &\hspace{1.5cm}-\frac{f(x_{i-1},y_{i-1})+f(x_{i+1},y_{i+1})\left( \Delta d_i+\Delta d_{i+1}\right)}
                                     {\Delta d_i+\Delta d_{i+1}}  , \\          
                                &\hspace{3.5cm}     i=1,...,N-1.
  \end{split}
\label{bah1}
\end{equation}
Using marks $P_i,\,Q_i,\,R_i,$ we have    
\begin{equation}
  \begin{split} 
     F_iy=&P_i+Q_i+R_i,\:i=1,2,\ldots,N-1,
  \end{split}
\label{bah1a}
\end{equation}
where 
\begin{equation}
  \begin{split} 
   P_i=\gamma\left(y_{i-1}-2y_i+y_{i+1} \right),
  \end{split}
\label{bah1b}
\end{equation}
and using Taylor expansions for $y_{i-1}$ and $y_{i+1},$ we get

\begin{equation}
  \begin{split} 
  Q_i
   &= \gamma q \left(y'_i\frac{h_i\sinh(\beta h_{i-1})-h_{i-1}\sinh(\beta h_i) }
                   {\sinh(\beta h_i)(\cosh(\beta h_{i-1})-1)+\sinh(\beta h_{i-1})(\cosh(\beta h_i)-1)} \right.    \\ 
   &\hspace{.5cm}+\frac{y''_i}{2}\cdot\frac{ h^2_{i-1}\sinh(\beta h_i)+h^2_{i}\sinh(\beta h_{i-1})}
              {\sinh(\beta h_i)(\cosh(\beta h_{i-1})-1)+\sinh(\beta h_{i-1})(\cosh(\beta h_i)-1)}  \\ 
   &\hspace{1.cm}-\frac{\varepsilon^2y''_i}{\gamma}\cdot\frac{
                \sinh(\beta h_i)(\cosh(\beta h_{i-1})-1)+\sinh(\beta h_{i-1})(\cosh(\beta h_i)-1) }
              {\sinh(\beta h_i)(\cosh(\beta h_{i-1})-1)+\sinh(\beta h_{i-1})(\cosh(\beta h_i)-1)}  \\               
   &\hspace{1.5cm}+\frac{y'''_i}{6}\cdot\frac{h^3_{i}\sinh(\beta h_{i-1})-h^3_{i-1}\sinh(\beta h_i)}
        {\sinh(\beta h_i)(\cosh(\beta h_{i-1})-1)+\sinh(\beta h_{i-1})(\cosh(\beta h_i)-1)}       \\    
   &\hspace{2cm}+\left.\frac{y^{(iv)}(\zeta^{-}_{i-1})h^4_{i-1}\sinh(\beta h_i)+y^{(iv)}(\zeta^{+}_i)h^4_i\sinh(\beta h_{i-1})}
               {24(\sinh(\beta h_i)(\cosh(\beta h_{i-1})-1)+\sinh(\beta h_{i-1})(\cosh(\beta h_i)-1))}\right),  
  \end{split}
\label{bah2}
\end{equation}

and 

\begin{equation}
   \begin{split} 
    R_i
  &=2\gamma y'_i\frac{h_i\sinh(\beta h_{i-1})-h_{i-1}\sinh(\beta h_i)} 
                       {\sinh(\beta h_i)(\cosh(\beta h_{i-1})-1)+\sinh(\beta h_{i-1})(\cosh(\beta h_i)-1)}  \\
  &\hspace{.5cm}+2\gamma y''_i\left( \frac{1}{2}\cdot
               \frac{h^2_{i-1}\sinh(\beta h_i)+h^2_i\sinh(\beta h_{i-1})   }
                       {\sinh(\beta h_i)(\cosh(\beta h_{i-1})-1)+\sinh(\beta h_{i-1})(\cosh(\beta h_i)-1)} \right. \\
  &\hspace{2cm}-\frac{1}{\beta^2}\cdot\left.
               \frac{ \sinh(\beta h_i)(\cosh(\beta h_{i-1})-1)+\sinh(\beta h_{i-1})(\cosh(\beta h_i)-1) }
                       {\sinh(\beta h_i)(\cosh(\beta h_{i-1})-1)+\sinh(\beta h_{i-1})(\cosh(\beta h_i)-1)}\right)  \\                                                                                                                                                         
  &\hspace{.75cm}+\frac{\gamma y'''_i}{3}\cdot \frac{h^3_i\sinh(\beta h_{i-1})-h^3_{i-1}\sinh(\beta h_i)}
                        {\sinh(\beta h_i)(\cosh(\beta h_{i-1})-1)+\sinh(\beta h_{i-1})(\cosh(\beta h_i)-1)} \\
  &\hspace{1cm}                         +y'''_i\varepsilon^2(h_{i-1}-h_{i}) \\
  &\hspace{1.25cm}+ 2\gamma \frac{\tfrac{y^{(iv)}(\zeta^{-}_{i-1})h^4_{i-1}}{24}\sinh(\beta h_i)
                                        +\tfrac{y^{(iv)}(\zeta^{+}_{i})h^4_{i}}{24}\sinh(\beta h_{i-1})} 
                      {\sinh(\beta h_i)(\cosh(\beta h_{i-1})-1)+\sinh(\beta h_{i-1})(\cosh(\beta h_i)-1)} \\
  &\hspace{1.5cm}               -\tfrac{\varepsilon^2}{2}\left[y^{(iv)}(\mu^{-}_{i-1})h^2_{i-1}+y^{(iv)}(\mu^{+}_{i})h^2_{i}\right]                                                 
   \end{split}
\label{bah7}
\end{equation}
$i=1,2,\ldots,N-1$ and $\zeta^{-}_{i-1},\mu^{-}_{i-1}\in\left(x_{i-1},x_i\right),\,\zeta^{+}_{i},\mu^{+}_i\in\left(x_i,x_{i+1}\right).$

\begin{lemma}Let the mesh size $h_{i}=x_{i+1}-x_{i},$ be defined by the generating function \eqref{bakh11}. It holds the estimate 
  \begin{equation}
     \begin{array}{c}
      \left|y_{i-1}-y_i-(y_i-y_{i+1}) \right|\leqslant C \left(\left|y'_i\right|(h_i-h_{i-1})
                   +\tfrac{\left|y''(\delta^{-}_{i-1})\right|}{2}h^2_{i-1}+ \tfrac{\left|y''(\delta^{+}_{i})\right|}{2}h^2_{i}\right),\\
                   i=1,\ldots,N/2,
     \end{array}
    \label{lema2}
  \end{equation}
 \label{lema41} 
where are $\delta^{-}_{i-1}\in\left(x_{i-1},x_i\right),\,\delta^{+}_{i}\in\left(x_{i},x_{i+1}\right).$
\end{lemma}
\begin{proof} The proof is trivial, using Taylor expansions for $y_{i-1}$ and  $y_{i+1}$ we obtain \eqref{lema2}.
\end{proof}

\begin{lemma} Let the mesh size $h_{i}=x_{i+1}-x_{i},$ be defined by the generating function \eqref{bakh11} and $\beta$ defined by \eqref{beta}. It holds estimate
   \begin{equation}
        \left|y'_i\right| \tfrac{\left|h_i\sinh(\beta h_{i-1})-h_{i-1}\sinh(\beta h_i) \right|}
                   {\sinh(\beta h_i)(\cosh(\beta h_{i-1})-1)+\sinh(\beta h_{i-1})(\cosh(\beta h_i)-1)}
                                      \leqslant C\left|y'_i\right|(h_{i}-h_{i-1}),\,i=1,\ldots,N/2.
     \label{lema3}
   \end{equation}       
 \label{lema42}  
\end{lemma}
\begin{proof} We have that 
\begin{equation}
  \begin{split} 
   &\left|y'_i\right| \frac{\left|h_i\sinh(\beta h_{i-1})-h_{i-1}\sinh(\beta h_i) \right|}
                   {\sinh(\beta h_i)(\cosh(\beta h_{i-1})-1)+\sinh(\beta h_{i-1})(\cosh(\beta h_i)-1)} \\
   &\hspace{.25cm}=\left|y'_i\right|
               \frac{\beta h_{i-1}h_i\sum_{n=1}^{+\infty}{ \tfrac{\beta^{2n}(h^{2n}_{i}-h^{2n}_{i-1})}{(2n+1)!}}}
                      {\sinh(\beta h_i)(\cosh(\beta h_{i-1})-1)+\sinh(\beta h_{i-1})(\cosh(\beta h_i)-1)}   \\  
    &\hspace{.25cm}              \leqslant\left|y'_i\right|
      \frac{\beta h_{i-1}h_i(h^2_{i}-h^2_{i-1}) \sum_{n=1}^{+\infty}{ \tfrac{\beta^{2n}h^{2n-2}_i}{(2n)!}}}
                     {\sinh(\beta h_i)(\cosh(\beta h_{i-1})-1)+\sinh(\beta h_{i-1})(\cosh(\beta h_i)-1)} \\
   &\hspace{.25cm}               \leqslant 2 \left| y'_i \right|
              \frac{\beta^2 h_{i-1}h_i(h_{i}-h_{i-1})\sum_{n=0}^{+\infty}{ \tfrac{\beta^{2n+1}h^{2n+1}_i}{(2n+2)!}}}
                     {4\sinh\tfrac{\beta h_{i-1}}{2}\sinh\tfrac{\beta h_i}{2}\sinh\tfrac{\beta h_{i-1}+\beta h_i}{2}}\\
   &\hspace{.25cm}      = 2\left| y'_i \right|
              \frac{\beta^2 h_{i-1}h_i(h_{i}-h_{i-1})\frac{\cosh(\beta h_i)-1}{\beta h_i}}
                     {4\sinh\tfrac{\beta h_{i-1}}{2}\sinh\tfrac{\beta h_i}{2}\sinh\tfrac{\beta h_{i-1}+\beta h_i}{2}} \\  
   &\hspace{.25cm}= \left| y'_i \right|
             \frac{\beta h_{i-1}(h_{i}-h_{i-1})\sinh^2\frac{\beta h_i}{2}}
                      {\sinh\tfrac{\beta h_{i-1}}{2}\sinh\tfrac{\beta h_i}{2}\sinh\tfrac{\beta h_{i-1}+\beta h_i}{2}}  
                  \leqslant  C\left|y'_i\right|(h_i-h_{i-1}).                      
  \end{split}
\label{bah3}
\end{equation}
\end{proof}
\begin{remark} It is true that $\sum_{n=0}^{+\infty}{ \tfrac{x^{2n+1}}{(2n+2)!}}=\tfrac{\cosh x-1}{x},\,\cosh x-1=2\sinh^2\frac{x}{2}$\\ 
       and  $\sinh x(\cosh y-1)+\sinh y(\cosh x-1)=4\sinh\frac{x}{2}\sinh\frac{y}{2}\sinh\frac{x+y}{2}.$
\end{remark}       

\begin{lemma} Let the mesh size $h_{i}=x_{i+1}-x_{i},$ be defined by the generating function \eqref{bakh11} and $\beta$ defined by \eqref{beta}. We have the following estimate
    \begin{equation}
      \begin{array}{c} 
      \left|y''_i\right|\left|\tfrac{\tfrac{1}{2}\left[ h^2_{i-1}\sinh(\beta h_i)+h^2_{i}\sinh(\beta h_{i-1})\right]
            -\tfrac{\varepsilon^2}{\gamma}\left[ \sinh(\beta h_i)(\cosh(\beta h_{i-1})-1)+\sinh(\beta h_{i-1})(\cosh(\beta h_i)-1)\right]}
              {\sinh(\beta h_i)(\cosh(\beta h_{i-1})-1)+\sinh(\beta h_{i-1})(\cosh(\beta h_i)-1)}\right|\\
              \leqslant C\left| y''_i\right|h^2_i,\:i=1,\ldots,N/2.
      \end{array}
      \label{lema4}
    \end{equation}
\label{lema43}
\end{lemma}

\begin{proof} We have that
\begin{equation}
   \begin{split} 
    &\left|y''_i\right|\left|\tfrac{\tfrac{1}{2}\left[ h^2_{i-1}\sinh(\beta h_i)+h^2_{i}\sinh(\beta h_{i-1})\right]
            -\tfrac{\varepsilon^2}{\gamma}\left[ \sinh(\beta h_i)(\cosh(\beta h_{i-1})-1)+\sinh(\beta h_{i-1})(\cosh(\beta h_i)-1)\right]}
              {\sinh(\beta h_i)(\cosh(\beta h_{i-1})-1)+\sinh(\beta h_{i-1})(\cosh(\beta h_i)-1)} \right|\\
    &\hspace{.25cm}= \left| y''_i\right|
                \left|  \frac{\sinh(\beta h_i)\left(\tfrac{h^2_{i-1}}{2}-\tfrac{\cosh(\beta h_{i-1})-1}{\beta^2} \right)
                                           +\sinh(\beta h_{i-1})\left(\tfrac{h^2_{i}}{2}-\tfrac{\cosh(\beta h_{i})-1}{\beta^2} \right)}
               {\sinh(\beta h_i)(\cosh(\beta h_{i-1})-1)+\sinh(\beta h_{i-1})(\cosh(\beta h_i)-1)} \right| \\ 
    &\hspace{.25cm}= \left| y''_i\right|\left[\frac{h^2_{i-1}\sinh(\beta h_i)\left(\tfrac{\beta^2h^2_{i-1}}{4!}+\tfrac{\beta^4h^4_{i-1}}{6!}+\cdots \right)     }
               {\sinh(\beta h_i)(\cosh(\beta h_{i-1})-1)+\sinh(\beta h_{i-1})(\cosh(\beta h_i)-1)} \right.   \\  
    &\hspace{2.6cm}  +\left.\frac{h^2_i \sinh(\beta h_{i-1})\left(\tfrac{\beta^2h^2_{i}}{4!}+\tfrac{\beta^4h^4_{i}}{6!}+\cdots \right)}
               {\sinh(\beta h_i)(\cosh(\beta h_{i-1})-1)+\sinh(\beta h_{i-1})(\cosh(\beta h_i)-1)}\right]\\               
   &\hspace{.25cm}\leqslant\left| y''_i\right|  \frac{h^2_{i-1}\sinh(\beta h_i)\left(\cosh(\beta h_{i-1})-1\right)
                                           +h^2_i \sinh(\beta h_{i-1})\left(\cosh(\beta h_i)-1\right)}
               {\sinh(\beta h_i)(\cosh(\beta h_{i-1})-1)+\sinh(\beta h_{i-1})(\cosh(\beta h_i)-1)}    \\                                                    
   &\hspace{.25cm}\leqslant\left| y''_i\right| h^2_i \frac{\tfrac{h^2_{i-1}}{h^2_i}\sinh(\beta h_i)\left(\cosh(\beta h_{i-1})-1\right)
                                           + \sinh(\beta h_{i-1})\left(\cosh(\beta h_i)-1\right)}
               {\sinh(\beta h_i)(\cosh(\beta h_{i-1})-1)+\sinh(\beta h_{i-1})(\cosh(\beta h_i)-1)}\\&\hspace{.25cm} \leqslant C\left| y''_i\right|h^2_i  .         
   \end{split}
\label{bah4}
\end{equation}
\end{proof}
\begin{remark} It is true that $\sum_{n=0}^{+\infty}{\frac{x^{2n+2}}{(2n+4)!}}=\frac{\cosh x-1-\tfrac{x^2}{2}}{x^2},$  
           $0<\frac{\frac{\cosh x-1-\tfrac{x^2}{2}}{x^2}}{\cosh x-1}<1,$  and \\
          $\frac{\cosh x-1-\tfrac{x^2}{2}}{x^2}\leqslant C(\cosh x-1).$
\end{remark}

\begin{lemma}Let the mesh size $h_{i}=x_{i+1}-x_{i},$ be defined by the generating function \eqref{bakh11} and $\beta$ defined by \eqref{beta}. We have the following estimate 
   \begin{multline}
     \left|y'''_i\right|\left|\frac{h^3_{i}\sinh(\beta h_{i-1})-h^3_{i-1}\sinh(\beta h_i)}
               {\sinh(\beta h_i)(\cosh(\beta h_{i-1})-1)+\sinh(\beta h_{i-1})(\cosh(\beta h_i)-1)} \right|\\
                     \leqslant C\left|y'''_i\right|\left(\varepsilon^2+h^2_{i-1}\right)(h_i-h_{i-1}),\:
                     i=1,\ldots,N/2.
   \label{lema5}
   \end{multline}
\label{lema44}
\end{lemma} 
\begin{proof}We have that
\begin{equation}
   \begin{split} 
       & \left|y'''_i\right|\left|\frac{h^3_{i}\sinh(\beta h_{i-1})-h^3_{i-1}\sinh(\beta h_i)}
               {\sinh(\beta h_i)(\cosh(\beta h_{i-1})-1)+\sinh(\beta h_{i-1})(\cosh(\beta h_i)-1)} \right|\\
        &\hspace{.25cm}\leqslant  \left|y'''_i\right|
                          \left|\tfrac{\beta h_{i-1}h_i(h_{i-1}+h_i)(h_{i}-h_{i-1})}
                                       {\beta h_i\tfrac{\beta^2h^2_{i-1}}{2}+\beta h_{i-1}\tfrac{\beta^2h^2_i}{2}}
                                +\tfrac{\beta^3h^3_{i-1}h^3_i\sum_{n=1}^{+\infty}\frac{\beta^{2n}(h^{2n}_{i}-h^{2n}_{i-1})}{(2n+3)!}}
                                {\sinh(\beta h_i)(\cosh(\beta h_{i-1})-1)+\sinh(\beta h_{i-1})(\cosh(\beta h_i)-1)}            \right|\\
        &\hspace{.25cm}=\left|y'''_i\right|
                         \Bigg|\tfrac{2}{\gamma}\varepsilon^2(h_i-h_{i-1})\\
        &\hspace{1cm}    +\left.\frac{\beta^3h^3_{i-1}h^3_i(h^2_{i}-h^2_{i-1})\sum_{n=0}^{+\infty}\frac{\beta^{2n+2}(n+1)h^{2n}_i}{(2n+5)!}}
                                {\sinh(\beta h_i)(\cosh(\beta h_{i-1})-1)+\sinh(\beta h_{i-1})(\cosh(\beta h_i)-1)}            \right|\\   
        &\hspace{.25cm}\leqslant \left|y'''_i\right|
                        \left|\tfrac{2}{\gamma}\varepsilon^2(h_i-h_{i-1})
                                 +\frac{\beta^3h^3_{i-1}h^3_i(h_{i}-h_{i-1})(h_{i-1}+h_i)\sum_{n=0}^{+\infty}\frac{\beta^{2n+2}h^{2n}_i}{(2n+4)!}}
                                {4\sinh\tfrac{\beta h_{i-1}}{2}\sinh\tfrac{\beta h_{i}}{2}\sinh\tfrac{\beta (h_{i-1}+h_i)}{2}}            \right|
                                \\  
        &\hspace{.25cm}\leqslant  \left|y'''_i\right|
                        \left|\tfrac{2}{\gamma}\varepsilon^2(h_i-h_{i-1})
                                 +2\frac{\beta h^3_{i-1}(h_{i}-h_{i-1})\sum_{n=0}^{+\infty}\frac{\beta^{2n+4}h^{2n+4}_i}{(2n+4)!}}
                                {4\sinh\tfrac{\beta h_{i-1}}{2}\sinh\tfrac{\beta h_{i}}{2}\sinh\tfrac{\beta (h_{i-1}+h_i)}{2}}            \right|\\
        &\hspace{.25cm}\leqslant  \left|y'''_i\right|
                        \left|\tfrac{2}{\gamma}\varepsilon^2(h_i-h_{i-1})
                        +h^2_{i-1}(h_i-h_{i-1})\tfrac{\tfrac{\beta h_{i-1}}{2}}{\sinh\tfrac{\beta h_{i-1}}{2}}
                                 \cdot\frac{\cosh(\beta h_i)-1-\tfrac{\beta^2h^2_i}{2}}
                                {\sinh\tfrac{\beta h_{i}}{2}\sinh\tfrac{\beta (h_{i-1}+h_i)}{2}}            \right|\\                        
        &\hspace{.25cm}\leqslant C\left|y'''_i\right|\left(\varepsilon^2+h^2_{i-1}\right)|h_i-h_{i-1}|  .
   \end{split}
\label{bah5}
\end{equation}
\end{proof}

\begin{remark}It is true that $\sum_{n=0}^{+\infty}\frac{\beta^{2n+4}h^{2n+4}_i}{(2n+4)!}=\cosh(\beta h_i)-1-\tfrac{\beta^2h^2_i}{2}$ 
          and\\ $0<\frac{\cosh(\beta h_i)-1-\tfrac{\beta^2h^2_i}{2}}
                                {\sinh\tfrac{\beta h_{i}}{2}\sinh\tfrac{\beta (h_{i-1}+h_i)}{2}} <2.$ 
\end{remark}                                
 
\begin{lemma}Let the mesh size $h_{i}=x_{i+1}-x_{i},$ be defined by the generating function \eqref{bakh11} and $\beta$ defined by \eqref{beta}. We have the following estimate 
   \begin{multline}
     \left|\frac{y^{(iv)}(\zeta^{-}_{i-1})h^4_{i-1}\sinh(\beta h_i)+y^{(iv)}(\zeta^{+}_i)h^4_i\sinh(\beta h_{i-1})}
                               {\sinh(\beta h_i)(\cosh(\beta h_{i-1})-1)+\sinh(\beta h_{i-1})(\cosh(\beta h_i)-1)}\right|\\
 \leqslant C\varepsilon^2\left(\left|y^{(iv)}(\zeta^{-}_{i-1})\right|h^2_{i-1}+\left|y^{(iv)}(\zeta^{+}_i)\right|h^2_{i} \right),\,i=1,\ldots,N/2.                                                 
     \label{lema6}
   \end{multline}
\label{lema45}
\end{lemma}
\begin{proof} We have that
\begin{equation}
  \begin{split} 
      &\left|\frac{y^{(iv)}(\zeta^{-}_{i-1})h^4_{i-1}\sinh(\beta h_i)+y^{(iv)}(\zeta^{+}_i)h^4_i\sinh(\beta h_{i-1})}
                               {\sinh(\beta h_i)(\cosh(\beta h_{i-1})-1)+\sinh(\beta h_{i-1})(\cosh(\beta h_i)-1)}\right|\\
      &\hspace{1.25cm}\leqslant 
                \frac{\left|y^{(iv)}(\zeta^{-}_{i-1})\right|h^4_{i-1}\sinh(\beta h_i)}
                               {\sinh(\beta h_i)\tfrac{\beta^2 h^2_{i-1}}{2}+\sinh(\beta h_{i-1})\tfrac{\beta^2h^2_{i}}{2}} 
                 +\frac{\left|y^{(iv)}(\zeta^{+}_i)\right|h^4_i\sinh(\beta h_{i-1})}
                               {\sinh(\beta h_i)\tfrac{\beta^2 h^2_{i-1}}{2}+\sinh(\beta h_{i-1})\tfrac{\beta^2h^2_{i}}{2}}      \\
      &\hspace{1.25cm}\leqslant 2
                \left(\frac{\left|y^{(iv)}(\zeta^{-}_{i-1})\right|h^4_{i-1}\sinh(\beta h_i)}
                               {\beta^2h^2_{i-1}\left( \sinh(\beta h_i)+\sinh(\beta h_{i-1})\right)} 
                 +\frac{\left|y^{(iv)}(\zeta^{+}_i)\right|h^4_i\sinh(\beta h_{i-1})}
                               {\beta^2\sinh(\beta h_{i-1})\left( h^2_{i-1}+h^2_{i} \right)   } \right)   \\  
      &\hspace{1.25cm}\leqslant C\varepsilon^2\left(\left|y^{(iv)}(\zeta^{-}_{i-1})\right|h^2_{i-1}+\left|y^{(iv)}(\zeta^{+}_i)\right|h^2_{i} \right).                                                 
  \end{split}
\label{bah6}
\end{equation}
\end{proof}

Let us continue with the following lemma that will be further used in the proof of the $\varepsilon$--uniform convergence theorem. In the lemma quite rough estimate of $F_i$, is given  but it is fairly enough for our needs.

\begin{lemma}Let the mesh size $h_{i}=x_{i+1}-x_{i},$ be defined by the generating function \eqref{bakh11} and $\beta$ defined by \eqref{beta}. We have the following estimate
\begin{equation}
  \begin{split} 
  \left|F_iy\right|\leqslant C\left[\left|s_{i-1}\right|\left( \tfrac{\varepsilon^2}{h^2_{i-1}}+\tfrac{\varepsilon^2}{h^2_i}+4\gamma+q+2\right) +
                       \frac{1}{N^2} \right]                        ,\:i=1,2,\ldots, N/2.
  \end{split}
 \label{konz1}
\end{equation}
\end{lemma}
\begin{proof} For $F_iy,\:i=1,2,\ldots,N/2$ holds
\begin{equation}
   \begin{split} 
    \left| F_iy\right|=&\left|\tfrac{\gamma}{\triangle d_i+\triangle d_{i+1}}\left\{\left[(q+1)a_i+d_i+\triangle d_{i+1} \right](y_{i-1}-y_i)\right.\right.\\
                               &\hspace{.2cm}
                                -\left[(q+1)a_{i+1}+d_{i+1}+\triangle d_{i} \right](y_{i}-y_{i+1})\\
                               &\hspace{.5cm}-\left.\tfrac{f(x_{i-1},y_{i-1})+qf(x_{i},y_{i})+f(x_{i+1},y_{i+1}) }{\gamma}
                                \left.\left(\triangle d_{i}+\triangle d_{i+1}\right)\right\}\right|\\
                      =&\left|\tfrac{\gamma}{\triangle d_i+\triangle d_{i+1}}\left\{\left[(q+1)a_i+d_i+\triangle d_{i+1} \right](y_{i-1}-y_i)\right.\right.\\
                                &\hspace{.2cm}  -\left[(q+1)a_{i+1}+d_{i+1}+\triangle d_{i} \right](y_{i}-y_{i+1}) \\
                               &\hspace{.5cm}-\left.\varepsilon^2\tfrac{y''_{i-1}+qy''_{i}+y''_{i+1} }{\gamma}
                                \left.\left(\triangle d_{i}+\triangle d_{i+1}\right)\right\}\right|\\          
               \leqslant&\gamma(q+2)\left|\tfrac{a_i(s_{i-1}-s_i)-a_{i+1}(s_i-s_{i+1})}{\triangle d_i+\triangle d_{i+1}} \right|
                          +\gamma\left|s_{i-1}-2s_i+s_{i+1}\right|+\varepsilon^2\left|s''_{i-1}+qs''_i+s''_{i+1}\right|\\
                      &\hspace{0.2cm}+\left|\tfrac{\gamma}{\triangle d_i+\triangle d_{i+1}}\left\{\left[(q+1)a_i+d_i+\triangle d_{i+1} \right](r_{i-1}-r_i)\right.\right.\\
                               &\hspace{.4cm} -\left[(q+1)a_{i+1}+d_{i+1}+\triangle d_{i} \right](r_{i}-r_{i+1}) 
                                -\left.\left.\varepsilon^2\tfrac{r''_{i-1}+qr''_{i}+r''_{i+1} }{\gamma}
                                \left(\triangle d_{i}+\triangle d_{i+1}\right)\right\}\right|  
\end{split}
 \label{konz2}
\end{equation}
For the layer component $s,$ due to  \eqref{slojna}, we have  

\begin{align*}\allowdisplaybreaks
  &\gamma(\tilde{q}+2)\left|\frac{a_i(s_{i-1}-s_i)-a_{i+1}(s_i-s_{i+1})}{\Delta d_i+\Delta d_{i+1}} \right|\\
  &\hspace{2cm} +\gamma\left|s_{i-1}-2s_i+s_{i+1}\right| +\varepsilon^2\left|s''_{i-1}+\tilde{q}s''_i+s''_{i+1}\right|\\
  & \hspace{.25cm} =\gamma (\tilde{q}+2)\left|\frac{\frac{1}{\sinh(\beta h_{i-1})}}{\frac{\cosh(\beta h_{i-1})-1}{\sinh(\beta h_{i-1})}
   +\frac{\cosh(\beta h_{i})-1}{\sinh(\beta h_{i})}} (s_{i-1}-s_i)\right.\\
   &\hspace{2cm}   
    \left.   -\frac{\frac{1}{\sinh(\beta h_{i})}}{\frac{\cosh(\beta h_{i-1})-1}{\sinh(\beta h_{i-1})}
                              +\frac{\cosh(\beta h_{i})-1}{\sinh(\beta h_{i})}} (s_{i}-s_{i+1})\right|   \\
    & \hspace{3cm}+\gamma\left|s_{i-1}-2s_i+s_{i+1}\right|+\varepsilon^2\left|s''_{i-1}+\tilde{q}s''_i+s''_{i+1}\right|    \\
    & \hspace{.25cm}   \leqslant \gamma(\tilde{q}+2)\left(\left| \frac{s_{i-1}-s_i}{\cosh(\beta h_{i-1})-1}\right|
                    +\left|\frac{s_i-s_{i+1}}{\cosh(\beta h_i)-1}\right|\right)\\
    &\hspace{3cm}                 + \gamma\left|s_{i-1}-2s_i+s_{i+1}\right|+\varepsilon^2\left|s''_{i-1}+\tilde{q}s''_i+s''_{i+1}\right|       \\
    &\hspace{.25cm}          \leqslant  C e^{-\frac{x_{i-1}}{\varepsilon}\sqrt{m}}\left( \frac{\varepsilon^2}{h^2_{i-1}}+\frac{\varepsilon^2}{h^2_i}+4\gamma+\tilde{q}+2\right)\stepcounter{equation}\tag{\theequation}\label{konz3} .
\end{align*}
Now, for the regular component $r,$ due to mesh sizes \eqref{korak1}, \eqref{korak2}, the estimate \eqref{regularna}, expansions \eqref{bah1a}, \eqref{bah1b} \eqref{bah2}, \eqref{bah7}, and Lemma \ref{lema41}-- Lemma \ref{lema45}, we have

\begin{multline}
   \left|\tfrac{\gamma}{\triangle d_i+\triangle d_{i+1}}\left\{\left[(q+1)a_i+d_i+\triangle d_{i+1} \right](r_{i-1}-r_i)\right.\right.\\
                                -\left.\left.\left[(q+1)a_{i+1}+d_{i+1}+\triangle d_{i} \right](r_{i}-r_{i+1}) 
                               - \varepsilon^2\tfrac{r_{i-1}+qr_{i}+r_{i+1} }{\gamma}
                                \left(\triangle d_{i}+\triangle d_{i+1}\right)\right\}\right| \leqslant \frac{C}{N^2}.
   \label{konz4}
\end{multline}

Using \eqref{konz3} and \eqref{konz4} completes the proof of the lemma.
\end{proof}

Now we can state and prove the main theorem on $\varepsilon$--uniform convergence.

\begin{theorem}\label{teoremabah1} The discrete problem \eqref{diskretni2} on the Bakhvalov--type mesh from Section 3  is uniformly convergent with respect to $\varepsilon$ and
\begin{equation*}
\max_{0\leqslant i\leqslant N}\left|y(x_i)-\overline{y}_i\right|\leq \frac{C}{N^2},
\end{equation*}
where $y$ is a solution of the problem \eqref{uvod1}--\eqref{uvod3},  $\overline{y}$ is the corresponding solution of \eqref{diskretni2} and $C>0$ is a constant independent  of  $N$ and $\varepsilon$.
\end{theorem}

\begin{proof}   From $(\ref{bah1a}),$ due to  Lemma \ref{lema41}, Lemma \ref{lema42}, Lemma \ref{lema43}, Lemma \ref{lema44} and  Lemma \ref{lema45} we have   
  \begin{multline}
    \left| G_iy\right|  
       \leqslant C\left( \left|y'_i\right|(h_i-h_{i-1})
                   +\tfrac{\left|y''(\delta^{-}_{i-1})\right|}{2}h^2_{i-1}+ \tfrac{\left|y''(\delta^{+}_{i})\right|}{2}h^2_{i}
                      +\left| y''_i\right|h^2_i\right.\\
            +\left|y'''_i\right|\left(\varepsilon^2+h^2_{i-1}\right)(h_i-h_{i-1})  
                      +\varepsilon^2\left|y^{(iv)}(\zeta^{-}_{i-1})+y^{(iv)}(\mu^{-}_{i-1})\right|h^2_{i-1}\\
                              +\left.\varepsilon^2\left|y^{(iv)}(\zeta^{+}_i)+y^{(iv)}(\mu^{+}_i)\right|h^2_{i} \right),\
                               i=1,2,\ldots N/2.
  \label{teo3}
  \end{multline}
The statement of the theorem for the regular component due to  \eqref{konz4} is proved. 

For the layer component $s,$ holds the estimate  \eqref{slojna}.  On the observed part of mesh, we have that $e^{-\frac{x}{\varepsilon}\sqrt{m}}\geqslant e^{-\frac{1-x}{\varepsilon}\sqrt{m}},$ now it is enough to estimate   $e^{-\frac{x}{\varepsilon}\sqrt{m}}$, on this part of mesh.  

We use a technique from \cite{bahvalov1969} and \cite{herceg1990,vulanovic1983,vulanovic1993}.\\

\textbf{Case I}\\ Let   $t_{i-1}\geqslant \tau,$ 
 we have that 

\begin{equation}
   \begin{split} 
    e^{-\frac{x_{i-1}}{\varepsilon}\sqrt{m}}\leqslant & e^{-\frac{\kappa(\tau)}{\varepsilon}\sqrt{m}}
                                         =e^{-a\sqrt{m}\frac{p-\sqrt[3]{\varepsilon}}{\sqrt[3]{\varepsilon}}}
                                         \leqslant e^{-\frac{C}{\sqrt[3]{\varepsilon}}}.
   \end{split}
 \label{lema7a}
\end{equation}
Now for the layer component $s$ due to \eqref{teo3}, holds
  \begin{equation}
    \begin{split} 
     & \left|s'_i\right|(h_i-h_{i-1})
                   +\tfrac{\left|s''(\delta^{-}_{i-1})\right|}{2}h^2_{i-1}+ \tfrac{\left|s''(\delta^{+}_{i})\right|}{2}h^2_{i}
                      +\left| s''_i\right|h^2_i\\
     &\hspace{2cm} +\left|s'''_i\right|\left(\varepsilon^2+h^2_{i-1}\right)(h_i-h_{i-1})
                 +\varepsilon^2\left|s^{(iv)}(\zeta^{-}_{i-1})+s^{(iv)}(\mu^{-}_{i-1})\right|h^2_{i-1}\\
     &\hspace{8cm}                         +\varepsilon^2\left|s^{(iv)}(\zeta^{+}_i)+s^{(iv)}(\mu^{+}_i)\right|h^2_{i}  \\
     &\hspace{.25cm}\leqslant C_1\frac{e^{-\frac{C}{\sqrt[3]{\varepsilon}}}}{\varepsilon}(h_i-h_{i-1})
                   +C_2\frac{e^{-\frac{C}{\sqrt[3]{\varepsilon}}}}{\varepsilon^2}h^2_i\\
     &\hspace{2.5cm}    +C_3\frac{e^{-\frac{C}{\sqrt[3]{\varepsilon}}}}{\varepsilon^3}\left(\varepsilon^2+h^2_{i-1}\right)(h_i-h_{i-1})
                   \leqslant \frac{C}{N^2}.
    \end{split}
   \label{teo3a}   
  \end{equation} 
\textbf{Case II}\\
Let  $t_{i-1}<\tau$ i $t_{i-1}\leqslant p-3h,\:h=1/N.$ 
From $t_{i-1}\leqslant p-3h\Leftrightarrow p-t_{i-1}\geqslant 3h\Leftrightarrow p-t_{i+1}\geqslant h$ and $p-t_{i-1}=p-t_{i+1}+2h,$  we have that   
\begin{equation}
  p-t_{i+1}\geqslant \frac{p-t_{i-1}}{3}.
\label{lema9}
\end{equation}
Also, there holds the following estimate
\begin{equation}
  \begin{split} 
   e^{-\frac{x_{i-1}}{\varepsilon}\sqrt{m}}=e^{-\frac{a\sqrt{m}t_{i-1}}{p-t_{i-1}}}\leqslant C e^{-\frac{a\sqrt{m}p}{p-t_{i-1}}}.
  \end{split}
 \label{lema9a}
\end{equation}
From the construction of the mesh \eqref{bakh11}, it implies
\begin{equation}
  \kappa^{(k)}(\tau)=\pi^{(k)}(\tau),\:k\in\left\{0,1,2\right\}
 \label{lema10}
\end{equation}  
and  
\begin{equation}
  \begin{split} 
   \kappa'''(t)-\pi'''(t)=&\frac{6a\varepsilon p}{(p-t)^4}-6\omega\geqslant \frac{6a\varepsilon p}{(p-\tau)^4}-6\omega
                         =  6\left( \frac{a p}{\sqrt[3]{\varepsilon}}-\omega\right),\:t\in\left[ \tau,p\right).
  \end{split}
\label{lema11}
\end{equation}
Hence, for sufficiently small $\varepsilon$, we have  
\begin{subequations}
\begin{equation}
   \frac{a p}{\sqrt[3]{\varepsilon}}-\omega>0,
  \label{lema12}
\end{equation} 
\begin{equation}
  \kappa'''(t)-\pi'''(t)>0,\:t\in\left[\tau,p\right),
  \label{lema13}
\end{equation}
and, due to \eqref{lema10} and  \eqref{lema13}, we get  
\begin{equation}
  \kappa^{(k)}(t)>\pi^{(k)}(t),\:k\in\left\{0,1,2\right\},\:t\in\left[\tau,p\right).
  \label{lema14}
\end{equation}
\end{subequations}

Now, because of \eqref{bakh11}, \eqref{lema9} and \eqref{lema14} we have  
\begin{subequations}
  \begin{equation}
     h_i\leqslant\int_{i/N}^{(i+1)/N}{\kappa'(t)\dif t}\leqslant\frac{\kappa'(t_{i+1})}{N}
          =\frac{a\varepsilon p}{(p-t_{i+1})^2}\cdot \frac{1}{N} \leqslant\frac{9a\varepsilon p}{(p-t_{i-1})^2}\cdot \frac{1}{N}
    \label{lema15}
  \end{equation}
and
  \begin{multline}
    h_i-h_{i-1}\leqslant\int_{i/N}^{(i+1)/N}\int_{t-1/N}^{t}{\kappa''(s)\dif s\dif t}\\\leqslant\frac{\kappa''(t_{i+1})}{N^2}
          =\frac{2 a\varepsilon p}{(p-t_{i+1})^3}\cdot\frac{1}{N^2}\leqslant\frac{54a\varepsilon p}{(p-t_{i-1})^3}\cdot\frac{1}{N^2} .               
    \label{lema16}
  \end{multline}
\end{subequations}
Finally, for the layer component $s$ from \eqref{teo3}, due to \eqref{lema9a}, \eqref{lema15} and \eqref{lema16}, we obtain the estimate

\begin{align}
  & \left|s'_i\right|(h_i-h_{i-1})
                   +\tfrac{\left|s''(\delta^{-}_{i-1})\right|}{2}h^2_{i-1}+ \tfrac{\left|s''(\delta^{+}_{i})\right|}{2}h^2_{i}
                      +\left| s''_i\right|h^2_i\nonumber\\
  &\hspace{1.5cm}+\left|s'''_i\right|\left(\varepsilon^2+h^2_{i-1}\right)(h_i-h_{i-1}) 
                      +\varepsilon^2\left|s^{(iv)}(\zeta^{-}_{i-1})+s^{(iv)}(\mu^{-}_{i-1})\right|h^2_{i-1}\nonumber\\
  &\hspace{7.5cm}    +\varepsilon^2\left|s^{(iv)}(\zeta^{+}_i)+s^{(iv)}(\mu^{+}_i)\right|h^2_{i}  \nonumber \\
  &\hspace{.25cm}\leqslant  C_1 \frac{e^{-\frac{x_{i-1}}{\varepsilon}\sqrt{m}}}{\varepsilon}(h_i-h_{i-1})+C_2 \frac{e^{-\frac{x_{i-1}}{\varepsilon}\sqrt{m}}}{\varepsilon^2}h^2_{i}\nonumber\\
  &\hspace{3.25cm}       + C_3\frac{e^{-\frac{x_{i-1}}{\varepsilon}\sqrt{m}}}{\varepsilon^3}\left(\varepsilon^2+h^2_{i-1}\right)(h_i-h_{i-1}) 
                        +C_4 \varepsilon^2 \frac{e^{-\frac{x_{i-1}}{\varepsilon}\sqrt{m}}}{\varepsilon^4}h^2_{i} \nonumber\\
  &\hspace{.25cm}\leqslant C_5\left[\frac{e^{-\frac{a\sqrt{m}p}{p-t_{i-1}}}}{\varepsilon}\cdot \frac{54a\varepsilon p}{(p-t_{i-1})^3}\cdot\frac{1}{N^2}
                  +2\frac{e^{-\frac{a\sqrt{m}p}{p-t_{i-1}}}}{\varepsilon^2}\cdot \frac{81a^2\varepsilon^2 p^2}{(p-t_{i-1})^4}\cdot \frac{1}{N^2}  \right.\nonumber\\
  &\hspace{1.5cm}+\frac{e^{-\frac{a\sqrt{m}p}{p-t_{i-1}}}}{\varepsilon^3}\cdot\left(\varepsilon^2
                         +  \frac{81a^2\varepsilon^2 p^2}{(p-t_{i-1})^4}\cdot\frac{1}{N^2}\right)                                   
                 \cdot \frac{54a\varepsilon p}{(p-t_{i-1})^3}\cdot\frac{1}{N^2} \nonumber\\
  &\hspace{7cm}    \left.             +4\varepsilon^2\frac{e^{-\frac{a\sqrt{m}p}{p-t_{i-1}}}}{\varepsilon^4}\cdot\frac{81a^2\varepsilon^2 p^2}{(p-t_{i-1})^4}                 
                           \cdot\frac{1}{N^2}\right] \nonumber \\
  &\hspace{.25cm}\leqslant C_6\left[ \frac{e^{-\frac{a\sqrt{m}p}{p-t_{i-1}}}}{(p-t_{i-1})^3}
                   +\frac{e^{-\frac{a\sqrt{m}p}{p-t_{i-1}}}}{(p-t_{i-1})^4}
                    +\frac{e^{-\frac{a\sqrt{m}p}{p-t_{i-1}}}}{(p-t_{i-1})^7}\cdot\frac{1}{N^2} \right]\cdot\frac{1}{N^2}                     \nonumber    \\
  &\hspace{.25cm}\leqslant\frac{C}{N^2}                 .
\label{lema17}
\end{align}

\noindent \textbf{Case III}\\
At the end, let $p-3h<t_{i-1}<\tau,$ there holds 

\begin{equation}
  \begin{split} 
     e^{-\frac{x_{i-1}}{\varepsilon}\sqrt{m}}=  e^{-\frac{a\sqrt{m}(p-3h)}{3h}}\leqslant C e^{-\frac{a\sqrt{m}p}{3h}}.
  \end{split}
\label{lema20}
\end{equation}
On the observed part of the mesh is $h_{i-1}\leqslant h_i,$ and for the mesh sizes $h_{i-1},$ we have  

\begin{multline}
    h_{i-1}=\kappa(t_i)-\kappa(t_{i-1})=\kappa'(\theta_{i-1})h
          =\frac{a\varepsilon ph}{(p-\theta_{i-1})^2}\\\geqslant\frac{a\varepsilon ph}{(p-(p-3h))^2}=\frac{a\varepsilon p}{9h},\:
          \theta_{i-1}\in[t_{i-1},t_i].
 \label{lema20a}
\end{multline}
Now, due to \eqref{lema20}, \eqref{lema20a} and  \eqref{konz1} we get

\begin{multline}
   \left| G_iy\right|\leqslant C_1\left[e^{-\frac{x_{i-1}}{\varepsilon}\sqrt{m}}\left( \tfrac{\varepsilon^2}{h^2_{i-1}}
                                  +\tfrac{\varepsilon^2}{h^2_i}+4\gamma+\tilde{q}+2\right)+\frac{1}{N^2} \right]  \\
    	             \leqslant C_2 \left[e^{-\frac{x_{i-1}}{\varepsilon}\sqrt{m}}\left( 2h^2 +4\gamma+\tilde{q}+2\right)+\frac{1}{N^2} \right]
                     \leqslant \frac{C}{N^2}.
\label{lema21}
\end{multline}

Case $3h\leqslant\sqrt[3]{\varepsilon}$ is proved in Case I and  Case II.

According to  \eqref{konz4}, \eqref{teo3}, \eqref{teo3a}, \eqref{lema17} and  \eqref{lema21}, the proof of the theorem is complete.
\end{proof}

\section{Numerical results}
 \label{sekcija5}
 \setcounter{theorem}{0}

In this section we present the numerical results to confirm the uniform accuracy of the discrete problem \eqref{diskretni1} using the mesh from Section \ref{sekcija3}. To demonstrate the efficiency of the method, we present two examples having boundary layers.

\begin{example} Consider the following problem from {\rm \cite{vulanovic1993}}
\begin{subequations}
 \begin{equation}
    \varepsilon^2y''=y-1\:\text{ for }x\in[0,1],
  \label{num1}
 \end{equation}
\begin{equation}
   y(0)=y(1)=0.
  \label{num2}
\end{equation}
\end{subequations}

The exact solution of this problem $\eqref{num1}-\eqref{num2}$ is given by
\begin{equation}
    y(x)=1-\frac{e^{-\frac{x}{\varepsilon}}+e^{-\frac{1-x}{\varepsilon}}}{1+e^{-\frac{1}{\varepsilon}}}.
  \label{num3}
\end{equation}
\end{example}

The appropriate system was solved using initial guess\\ $y0=\left(0,y0_1,\ldots,y0_{N-1},0\right)^T,\:y0_i=1,\:i=1,\ldots,N-1,$ by Newton's method. 
The value of the constant $\gamma=1$ has been chosen  so that the condition $\gamma\geqslant f_y(x,y),\:\forall(x,y)\in[0,1]\times[y_L,y_U]\subset[0,1]\times\mathbb{R}$ is fulfilled, where $y_L$ and $y_U$ are the lower and upper solutions of the test problem \eqref{num1}--\eqref{num2} and their values are $y_L=0$ and $y_U=1.$  

Because of the fact that the exact solution is known, we define the computed error $E_n$ and the computed rate of convergence Ord in the usual way
\begin{equation}
  E_N=\left\|y-\overline{y}^N\right\|_{\infty}
 \label{num4}
\end{equation}
and
\begin{equation}
    \ord=\frac{\ln E_N-\ln E_{2N}}{\ln 2}.
 \label{num5}
\end{equation}
Other values of constants are $m=1,\:q=4,\:a=1$ and $p=0.4.$ \\
The values of $E_N$ and Ord are given at the of the paper in Appendix (Table \ref{tabela1}). 

\example\upshape Consider the following problem 
\begin{subequations}
\begin{equation}
    \varepsilon^2y''=(y-1)(1+(y-1)^2)\:\text{ for } x\in[0,1],
 \label{num1a}
\end{equation}
\begin{equation}
 y(0)=y(1)=0,
 \label{num2a}
\end{equation}
\end{subequations}

The exact solution of the problem $\eqref{num1a}-\eqref{num2a}$ is unknown.

The system of nonlinear equations is solved by Newton's method with initial guess $y0=\left(0,y0_1,\ldots,y0_{N-1},0\right)^T,\:y0_i=1,\:i=1,\ldots,N-1.$ 

The value of the constant $\gamma$ has been chosen so that local version of the condition $\eqref{uvod3}$ is fulfilled.
In other words, because  $f_y(x,y)\geqslant m>0,\forall (x,y)\in[0,1]\times[y_L,y_U],$  and in our example $y_L=0,\:y_U=1$, and $1\leqslant f_y(x,y)\leqslant 4,\:\forall (x,y)\in[0,1]\times[y_L,y_U],$  we get that  $\gamma=4.$

Because the fact that we do not know the exact solution, we will replace $y$ by $\hat{y}$ in order to calculate $E_N$ and Ord,  where $\hat{y}$ is the numerical solution of $\eqref{num1a}-\eqref{num2a}$  obtained by using $N=16384,$ (same procedure is applied in \cite{herceg1998}, \cite{herceg2000themesh}).

Now, we calculate the value of error $E_N$ and the the rate of convergence on the following way
\begin{equation}
  E_N=\left\|\hat{y}-\overline{y}^N\right\|_{\infty}
 \label{num4a}
\end{equation}
and
\begin{equation}
    \ord=\frac{\ln E_N-\ln E_{2N}}{\ln 2}.
 \label{num5a}
\end{equation}

Other values of constants are $m=1,\:q=4,\:a=1$ and $p=0.3.$ \\
The values of $E_N$ and Ord are given at the of the paper in Appendix (Table \ref{tabela2}).

\section{Conclusion}
In this paper we presented a discretization of an one--dimensional semilinear reaction--diffusion problem, with suitable assumptions that have ensured the existence and uniqueness of the continuous problem. We constructed a class of differential schemes, and we proved the existence and uniqueness of the numerical solution, after which we proved the $\varepsilon$--uniform convergence using a suitable layer--adaptive mesh. Finally we performed a numerical experiments to confirm the theoretical results.\\\\

\noindent\textbf{Acknowledgement}\textit{.} This paper is the part of Project "Numeri\v cko rje\v savanje kvazilinearnog singularno--perturbacionog jednodimenzionalnog rubnog problema". The paper has emanated from research conducted with the partial financial support of  Ministry of education and sciences of Federation of Bosnia and Herzegovina and University of Tuzla under grant 01/2-3995-V/17 of 18.12.2017.  Helena Zarin is supported by the Ministry of Education and Science of the Republic of Serbia under grant no. 174030.  The authors are grateful to Sanela Halilovi\' c and Rifat Omerovi\' c for helpful advice.

\newpage
\section*{Appendix}
\begin{table}[h]
\centering
\begin{tabular}{c|cc|cc|cc|}\hline
     $N$ &$E_n$&Ord&$E_n$&Ord&$E_n$&Ord\\\hline     
$2^{6}$&$2.4133e-04$&$2.00$   &$1.0062e-03$&$1.98$    &$1.3294e-03$&$1.96$        \\      
$2^{7}$&$6.0436e-05$&$2.00$   &$2.5429e-04$&$1.99$    &$3.4128e-04$&$1.99$        \\
$2^{8}$&$1.5095e-05$&$2.00$   &$6.3802e-05$&$2.00$    &$8.5934e-05$&$2.00$      \\
$2^{9}$&$3.7691e-06$&$1.97$   &$1.5961e-05$&$2.00$    &$2.1523e-05$&$2.00$         \\
$2^{10}$&$9.6433e-07$&$2.00$  &$3.9909e-06$&$2.00$    &$5.3833e-06$&$2.00$       \\
$2^{11}$&$2.4108e-07$&$2.00$  &$9.9777e-07$&$2.00$    &$1.3460e-06$&$2.00$        \\
$2^{12}$&$6.0271e-08 $&$2.00$ &$2.4945e-07$&$2.00$    &$3.3651e-07$&$2.00$      \\
$2^{13}$&$1.5068e-08 $&$-$    &$6.2363e-08$&$-$       &$8.4127e-08$&$-$         \\\hline
 $\varepsilon$&\multicolumn{2}{c}{$2^{-3}$}&\multicolumn{2}{c}{$2^{-5}$}
            &\multicolumn{2}{c}{$2^{-10}$}\\\hline\hline
     $N$ &$E_n$&Ord&$E_n$&Ord&$E_n$&Ord\\\hline    
$2^{6}$  &$1.3243e-03$&$1.96$ &$1.3243e-03$&$1.96$ &$1.3243e-03$&$1.96$       \\      
$2^{7}$  &$3.3945e-04$&$1.99$ &$3.3945e-04$&$1.99$ &$3.3945e-04$&$1.99$       \\
$2^{8}$  &$8.5413e-05$&$2.00$ &$8.5413e-05$&$2.00$ &$8.5413e-05$&$2.00$        \\
$2^{9}$  &$2.1388e-05$&$2.00$ &$2.1388e-05$&$2.00$ &$2.1388e-05$&$2.00$         \\
$2^{10}$ &$5.3493e-06$&$2.00$ &$5.3493e-06$&$2.00$ &$5.3493e-06$&$2.00$       \\
$2^{11}$ &$1.3375e-06$&$2.00$ &$1.3375e-06$&$2.00$ &$1.3375e-06$&$2.00$        \\
$2^{12}$ &$3.3438e-07$&$2.00$ &$3.3438e-07$&$2.00$ &$3.3438e-07$&$2.00$       \\
$2^{13}$ &$8.3595e-08$&$-$    &$8.3595e-08$&$-$   &$8.3595e-08$&$-$         \\\hline\hline
 $\varepsilon$&\multicolumn{2}{c}{$2^{-15}$}
            &\multicolumn{2}{c}{$2^{-25}$}&\multicolumn{2}{c}{$2^{-30}$}\\\hline
            $N$ &$E_n$&Ord&$E_n$&Ord&$E_n$&Ord\\\hline 
$2^{6}$       &$1.3243e-03$&$1.96$  &$1.3243e-03$&$1.96$   &$1.3243e-03$&$1.96$   \\      
$2^{7}$       &$3.3945e-04$&$1.99$ &$3.3945e-04$&$1.99$    &$3.3945e-04$&$1.99$    \\
$2^{8}$       &$8.5413e-05$&$2.00$  &$8.5413e-05$&$2.00$    &$8.5413e-05$&$2.00$   \\
$2^{9}$      &$2.1388e-05$&$2.00$  &$2.1388e-05$&$2.00$    &$2.1388e-05$&$2.00$     \\
$2^{10}$     &$5.3493e-06$&$2.00$  &$5.3493e-06$&$2.00$   &$5.3493e-06$&$2.00$   \\
$2^{11}$      &$1.3375e-06$&$2.00$  &$1.3375e-06$&$2.00$   &$1.3375e-06$&$2.00$   \\
$2^{12}$     &$3.3438e-07$&$2.00$  &$3.3438e-07$&$2.00$   &$3.3438e-07$&$2.00$   \\
$2^{13}$         &$8.3595e-08$&$-$     &$8.3595e-08$&$-$       &$8.3595e-08$&$-$      \\\hline
 $\varepsilon$&\multicolumn{2}{c}{$2^{-35}$}
            &\multicolumn{2}{c}{$2^{-40}$}&\multicolumn{2}{c}{$2^{-45}$}\\\hline
\end{tabular}
\caption{Error $E_N$ and convergence rates Ord for approximate solution.}
\label{tabela1}
\end{table}
\newpage
\begin{table}[h]
\centering
\begin{tabular}{c|cc|cc|cc|cc|cc}\hline
     $N$ &$E_n$&Ord&$E_n$&Ord&$E_n$&Ord\\\hline     
$2^{6}$&$5.4721e-04$&$1.99$   &$2.4570e-03$&$1.96$    &$2.7149e-03$&$1.93$         \\      
$2^{7}$&$1.3773e-04$&$2.00$   &$6.3358e-04$&$1.99$    &$7.1008e-04$&$1.98$        \\
$2^{8}$&$3.4477e-05$&$2.00$   &$1.5962e-04$&$2.00$    &$1.7971e-04$&$2.00$        \\
$2^{9}$&$8.6159e-06$&$2.00$   &$3.9985e-05$&$2.00$    &$4.5044e-05$&$2.00$          \\
$2^{10}$&$2.1478e-06$&$2.02$  &$9.9654e-06$&$2.02$    &$1.1237e-05$&$2.02$        \\
$2^{11}$&$5.3066e-07$&$2.07$  &$2.4624e-06$&$2.07$    &$2.7767e-06$&$2.07$        \\
$2^{12}$&$1.2635e-07$&$2.07$  &$5.8629e-07 $&$2.07$    &$6.6115e-07$&$2.07$       \\
$2^{13}$&$3.0091e-08 $&$-$    &$1.3963e-07 $&$-$      &$1.5746e-07$&$-$              \\\hline
 $\varepsilon$&\multicolumn{2}{c}{$2^{-3}$}&\multicolumn{2}{c}{$2^{-5}$}
            &\multicolumn{2}{c}{$2^{-10}$}\\\hline\hline
    $N$ &$E_n$&Ord&$E_n$&Ord&$E_n$&Ord\\\hline     
$2^{6}$  &$1.2281e-03$&$1.80$ &$1.2276e-03$&$1.80$ &$1.2276e-03$&$1.80$        \\      
$2^{7}$  &$3.5293e-04$&$1.94$  &$3.5247e-04$&$1.94$ &$3.5247e-04$&$1.94$       \\
$2^{8}$  &$9.1947e-05$&$1.98$  &$9.1749e-05$&$1.98$ &$9.1749e-05$&$1.98$        \\
$2^{9}$  &$2.3235e-05$&$2.00$ &$2.3180e-05$&$2.00$ &$2.3180e-05$&$2.00$       \\
$2^{10}$ &$5.8267e-06$&$2.00$  &$5.8140e-06$&$2.00$ &$5.8140e-06$&$2.00$        \\
$2^{11}$ &$1.4577e-06$&$2.00$  &$1.4544e-06$&$2.00$ &$1.4544e-06$&$2.00$        \\
$2^{12}$ &$3.6449e-07$&$2.00$  &$3.6366e-07$&$2.00$ &$3.6366e-07$&$2.00$       \\
$2^{13}$ &$9.1127e-08$&$-$ &$9.0921e-08$&$-$    &$9.0921e-08$&$-$            \\\hline
 $\varepsilon$&\multicolumn{2}{c}{$2^{-15}$}
             &\multicolumn{2}{c}{$2^{-25}$}&\multicolumn{2}{c}{$2^{-30}$}\\\hline\hline
     $N$ &$E_n$&Ord&$E_n$&Ord&$E_n$&Ord\\\hline     
$2^{6}$       &$1.2276e-03$&$1.80$ &$1.2276e-03$&$1.80$    &$1.2276e-03 $&$1.80 $   \\      
$2^{7}$       &$3.5247e-04$&$1.94$ &$3.5247e-04$&$1.94$    &$3.5247e-04 $&$1.94 $   \\
$2^{8}$       &$9.1749e-05$&$1.98$ &$9.1749e-05$&$1.98$    &$9.1749e-05 $&$1.98 $   \\
$2^{9}$       &$2.3180e-05$&$2.00$ &$2.3180e-05$&$2.00$    &$2.3180e-05 $&$2.00 $   \\
$2^{10}$      &$5.8140e-06$&$2.00$ &$5.8140e-06$&$2.00$    &$5.8140e-06 $&$2.00 $   \\
$2^{11}$      &$1.4544e-06$&$2.00$ &$1.4544e-06$&$2.00$    &$1.4544e-06 $&$2.00 $   \\
$2^{12}$      &$3.6367e-07$&$2.00$ &$3.6367e-07$&$2.00$    &$3.6367e-07 $&$2.00 $   \\
$2^{13}$      &$9.0921e-08$&$-$    &$9.0921e-08$&$-$       &$9.0921e-08$&$-$      \\\hline
 $\varepsilon$&\multicolumn{2}{c}{$2^{-35}$}
             &\multicolumn{2}{c}{$2^{-40}$}&\multicolumn{2}{c}{$2^{-45}$}\\\hline             
\end{tabular}
\caption{Error $E_N$ and convergence rates Ord for approximate solution.}
\label{tabela2}
\end{table}

\noindent Samir Karasulji\' c,\\
Faculty of Natural Sciences and Mathematics, \\University of Tuzla, \\Univerzitetska ulica 4, 75000 Tuzla, Bosnia and Herzegovina \\
E-mail: {\texttt{ samir.karasuljic@untz.ba}
}
\\\\
\noindent Helena Zarin,\\
Department of Mathematics and Informatics, \\Faculty of Natural Sciences and Mathematics, \\University of Novi Sad,\\ Trg Dositeja Obradovi\' ca, 21000 Novi Sad, Serbia\\
E-mail: {\texttt{helena.zarin@dmi.uns.ac.rs}
}
\\\\
\noindent Enes Duvnjakovi\' c,\\
Department of Mathematics,\\ Faculty of Natural Sciences and Mathematics,\\ University of Tuzla, \\Univerzitetska ulica 4, 75000 Tuzla, Bosnia and Herzegovina\\
E-mail: {\texttt{enes.duvnjakovic@untz.ba}


\begin{thebibliography}{00}

   \bibitem{bahvalov1969} {\sc N.S. Bakhvalov }, Towards optimization of methods for solving boundary value problems in the presence of boundary layers,  
                          \emph{ Zh Vychisl Mat  Mat Fiz}, {\bf 9} (1969),  841--859.
                          
   \bibitem{boglaev1984approximate} {\sc I. P. Boglaev}, Approximate solution of a non-linear boundary value problem with a small parameter for the highest-order differential, 
                              \emph {Zh Vychisl Mat  Mat Fiz}, {\bf 24}(11), (1984), 1649--1656.
                              
   \bibitem{samir2010scheme} {\sc E. Duvnjakovi\'{c}, S. Karasulji\'{c}, N. Oki\v{c}i\'{c}}, Difference Scheme for Semilinear Reaction-Diffusion Problem. 
                          \emph{ In: 14th International Research/Expert Conference "Trends in the Development of Machinery and Associated Technology"
                          TMT 2010;  11-18 September 2010; Mediterranean Cruise. Zenica: University of Zenica,} (2010), 793--796.
                                   
                                   
   \bibitem{samir2011scheme} {\sc E. Duvnjakovi\'c, S. Karasulji\'c}, Difference Scheme for Semilinear Reaction-Diffusion Problem on a Mesh of Bakhvalov Type,  
                               \emph{Math Balkanica}, {\bf 25}(5), (2011), 499--504.
                               
   \bibitem{samir2011skoplje} {\sc E. Duvnjakovi\'{c},  S. Karasulji\'{c}}, Uniformly Convergente Difference Scheme for Semilinear Reaction-Diffusion Problem,
                              \emph{ In: SEE Doctoral Year Evaluation Workshop, Skopje, Macedonia}, (2011).
                              
   \bibitem{samir2012uniformnly} {\sc E. Duvnjakovi\'{c}, S. Karasulji\'{c}}, Difference Scheme for Semilinear Reaction-Diffusion Problem, 
                             \emph{The Seventh Bosnian-Herzegovinian Mathematical Conference, Sarajevo, BiH}, (2012).
 
  \bibitem{samir2012class} {\sc E. Duvnjakovi\'{c}, S. Karasulji\'{c}}, Class of Difference Scheme for Semilinear Reaction-Diffusion Problem on Shishkin Mesh,
                           \emph{MASSEE International Congress on Mathematics - MICOM 2012, Sarajevo, Bosnia and Herzegovina}, (2012).
                           
  \bibitem{samir2013collocation} {\sc E. Duvnjakovi\'{c}, S.  Karasulji\'{c}}, Collocation Spline Methods for Semilinear Reaction-Diffusion Problem on Shishkin Mesh,
                         \emph{IECMSA-2013, Second International Eurasian Conference on Mathematical Sciences and Applications, Sarajevo, Bosnia and Herzegovina}, (2013).
                         
  
   \bibitem{samir2015uniformlyconvergent} {\sc E. Duvnjakovi{\'c}, S. Karasulji{\'c}, V. Pa{\v s}i{\'c}, H. Zarin}, 
                 A uniformly convergent difference scheme on a modified Shishkin mesh for the singularly perturbed reaction-diffusion boundary value problem,  
                 \emph{J Mod Meth Numer Math}, {\bf 6}(1),  (2015), 28--43.                            
  
   \bibitem{herceg1990} {\sc D. Herceg}, Uniform fourth order difference scheme for a singular perturbation problem, \emph{ Numer Math}, {\bf 56}(7), (1989), 675--693.
   
   \bibitem{herceg1991} {\sc D. Herceg, K. Surla}, Solving a nonlocal singularly perturbed problem by spline in tension, \emph{ Novi Sad J Math}, {\bf 21}(2), (1991), 119-132.

   \bibitem{herceg1998} {\sc D. Herceg, K. Surla, S. Rapaji\'{c}}, Cubic spline difference scheme on a mesh of a Bakhvalov type,   
                                \emph{Novi Sad J Math}, {\bf 28}(3), (1998), 41--49.
                                
                                
   \bibitem{herceg2000} {\sc D. Herceg, H. Mali\v{c}i\'c, I. Liki\'c}, On a finite difference analogue of fourth order for boundary value problem,  
                               \emph{Novi Sad J Math}, {\bf 30}(1), (2000), 197--203.
                               
                               
   \bibitem{herceg2000themesh} {\sc D. Herceg,  H. Mali{\v c}i{\' c}}, The Mesh Chopping Algorithm for Singularly Perturbed Boundary Value Problem,  
                                    \emph{Novi Sad J Math}, {\bf 30}(1), (2000), 155--164.
                                
                                    
   \bibitem{herceg2003} {\sc D. Herceg, M. Miloradovi\'{c}}, On numerical solution of semilinear singular perturbation problems by using the Hermite scheme on 
                            a new Bakhvalov-type mesh, \emph{ Novi Sad J Math}, {\bf 33}(1),  (2003), 145--162.
                            
                            
   \bibitem{herceg2003a} {\sc D. Herceg, {Dj}. Herceg}, On a fourth-order finite difference method for nonlinear two-point boundary value problems,  
                           \emph{Novi Sad J Math}, {\bf 33}(2), (2003), 173--180.
   
   \bibitem{samir2015construction} {\sc S. Karasulji{\'{c}}, E. Duvnjakovi{\'{c}}}, Construction of the Difference Scheme for Semilinear Reaction-Diffusion 
                        Problem on a Bakhvalov Type Mesh, \emph{In:The Ninth Bosnian-Herzegovinian Mathematical Conference, Sarajevo, BiH}, (2015).                 
             
                 
    
   \bibitem{samir2015uniformly} {\sc S. Karasulji{\'c}, E. Duvnjakovi{\'c}, H. Zarin}, Uniformly convergent difference scheme for a semilinear reaction-diffusion 
                      problem, \emph{ Adv Math Sci J}, {\bf 4}(2), (2015), 139--159.                           

   \bibitem{samir2017construction} {\sc S. Karasulji{\' c}, E. Duvnjakovi{\' c}, V. Pa{\v s}i{\' c}, E. Barakovi{\' c}, E.}, Construction of a global
             solution for the one dimensional singularly--perturbed boundary value problem, \emph{Journal of Modern Methods in Numerical Mathematics}, 
             {\bf 8}(1--2), (2017), 52--65.
                  
  \bibitem{samir2018uniformly} {\sc S. Karasulji{\' c}, E. Duvnjakovi{\' c}, E. Memi{\' c}}, Uniformly Convergent Difference Scheme for 
                    a Semilinear Reaction-Diffusion Problem on a Shishkin Mesh, \emph{Advances in Mathematics: Scientific Journal}, 
                    {\bf 7}, (1), (2018), 23--38.
    
  
   \bibitem{kopteva2001} {\sc  N. Kopteva, T. Lin{\ss}}, Uniform second-order pointwise convergence of a central difference approximation 
              for a quasilinear convection-diffusion problem, \emph{J Comput Appl Math}, {\bf 137}(2), (2001), 257--267.
              
   \bibitem{kopteva2001robust} {\sc N. Kopteva, M. Stynes}, A robust adaptive method for a quasi-linear one-dimensional convection-diffusion problem, 
                                 \emph{ SIAM J Numer Anal}, {\bf 39}(4), (2001), 1446--1467. 
               
   \bibitem{kopteva2004} {\sc N. Kopteva, M. Stynes}, Numerical analysis of a singularly perturbed nonlinear reaction--diffusion problem with multiple solutions,  
                              \emph{Appl Numer Math}, {\bf 51}(2), (2004), 273--288.
                       
   \bibitem{kopteva2007maximum} {\sc N. Kopteva}, Maximum norm error analysis of a 2d singularly perturbed semilinear reaction-diffusion problem, \emph{ Math Comp},
                                  {\bf 76}(258), (2007), 631--646.
                              
                              
                              
   \bibitem{kopteva2009} {\sc N. Kopteva, M. Pickett, H. Purtill}, A robust overlapping Schwarz method for a singularly perturbed semilinear reaction--diffusion 
                                problem with multiple solutions.  \emph{Int J Numer Anal Model}, {\bf 6}, (2009), 680--695.
               
               
   \bibitem{linss2000uniform} {\sc T. Lin{\ss}, H. G. Roos, R. Vulanovi{\'c}},  Uniform pointwise convergence on Shishkin-type meshes for quasi-linear 
                         convection-diffusion problems.  \emph{SIAM J Numer Anal}, {\bf 38}(3), (2000), 897--912.
           
           
   \bibitem{linss2012approximation} {\sc T. Lin{\ss}, G. Radojev, H. Zarin}, Approximation of singularly perturbed reaction-diffusion problems by quadratic $C^1$-splines.
                          \emph{Numer Algorithms}, {\bf 61}(1), (2012), 35--55.
                        
                        
   \bibitem{lorenz1982stability} {\sc J. Lorenz}, Stability and monotonicity properties of stiff quasilinear boundary problems, \emph{ Novi Sad J Math}, {\bf 12}, 
                                         (1982), 151--176.
    
    
   \bibitem{niijima} {\sc K. Niijima}, A uniformly convergent difference scheme for a semilinear singular perturbation problem, \emph{Numer Math}, {\bf 43}(2),
                            (1984), 175--198.
                                
   
   \bibitem{ortega2000} {\sc J. M. Ortega, W.C. Rheinboldt}, {\it Iterative Solution of Nonlinear Equations in Several Variables},   Philadelphia, PA, USA, SIAM, (2000).
    
   
   \bibitem{riordan1986uniformly} {\sc E. O'Riordan, M. Stynes},  A uniformly accurate finite-element method for a singularly perturbed one-dimensional reaction-diffusion
                       problem, \emph{ Math Comp}, {\bf 47}(176), (1986),  555--570.
   
   
                                 
   \bibitem{riordan1986uniform} {\sc E. O'Riordan, M. Stynes}, $L^1$ and $L^{\infty}$ Uniform Convergence of a Difference Scheme for a Semilinear Singular 
                     Perturbation Problem, \emph{ Numer Math}, {\bf 50}, (1986/87), 519-532.                
                     
   \bibitem{roos1990} {\sc H. G. Roos}, Global uniformly convergent schemes for a singularly perturbed boundary-value problem using patched base spline-functions, 
                          \emph{J Comput  Appl Math}, {\bf 29}(1), (1990), 69--77.
                    
   \bibitem{shishkin1988grid} {\sc G. I. Shishkin}, Grid approximation of singularly perturbed parabolic equations with internal layers, \emph{Sov J Numer Anal M Russ J  
                Numer Anal Math Model}, {\bf 3}(5), (1988), 393--408.
                
   \bibitem{stynes1987adaptive} {\sc M. Stynes}, An Adaptive Uniformly Convergent Numerical Method for a Semilinear Singular Perturbation Problem,   
                                      \emph{SIAM J Numer Anal}, {\bf 26}(2), (1987), 442-455.
             
   \bibitem{stynes1996} {\sc G. Sun, M. Stynes}, A uniformly convergent method for a singularly perturbed semilinear reaction-diffusion problem with multiple solutions, 
                             \emph{Math Comp}, {\bf 65}(215), (1996), 1085--1109.
                         
   \bibitem{stynes2006numerical} {\sc M. Stynes, N. Kopteva}, Numerical analysis of singularly perturbed nonlinear reaction-diffusion problems with multiple solutions,  
                                  \emph{Comput Math Appl}, {\bf 51}(5), (2006), 857--864.
                    
                    
   \bibitem{surla1996} {\sc K. Surla, Z. Uzelac}, A uniformly accurate difference scheme for singular perturbation problem,  
                         \emph{Indian J Pure Appl Mat}, {\bf 27}(10), (1996), 1005--1016.
                  
                  
   \bibitem{surla2003}{\sc K. Surla, Z. Uzelac}, On Stability of Spline Difference Scheme for Reaction-Diffusion Time-Dependent Singularly Perturbed Problem, 
                                    \emph{Novi Sad J Math}, {\bf 33}(2), (2003), 89-94.
                                
   \bibitem{vulanovic1983}{\sc R.  Vulanovi\'c}, On a numerical solution of a type of singularly perturbed boundary value problem by using a special discretization mesh, 
                               \emph{ Novi Sad J Math}, {\bf 13}, (1983), 187--201.            
  
   \bibitem{vulanovic1989}{\sc R.  Vulanovi\'c}, Mesh generation methods for numerical solution of quasilinear singular perturbation problems,   
                                \emph{Novi Sad J Math}, {\bf 19}(2), (1989), 171--193.
                 
   \bibitem{vulanovic1991second}{\sc R.  Vulanovi{\'c}}, A second order numerical method for non-linear singular perturbation problems without turning points,   
                                \emph{USSR Comp Math}, {\bf 31}(4), (1991), 522--532.
                          
   \bibitem{vulanovic1993} {\sc R. Vulanovi\'c}, On numerical solution of semilinear singular perturbation problems by using the Hermite scheme,  
                               \emph{Novi Sad J Math}, {\bf 23}(2), (1993), 363--379.
                          
   \bibitem{vulanovic2004}{\sc R. Vulanovi\'c}, An almost sixth-order finite-difference method for semilinear singular perturbation problems, 
                             \emph{Comput Methods Appl Math}, {\bf 4}(3), (2004), 368--383.
               
\end{thebibliography}
\end{document}